\documentclass[11pt,leqno]{amsart}
\usepackage{amscd}
\usepackage{color}
\usepackage{amssymb}
\usepackage{amsfonts}
\usepackage{latexsym}
\usepackage{verbatim}
\usepackage{enumitem}

\theoremstyle{plain}
\newtheorem{theorem}{Theorem}[section]
\newtheorem{definition}[theorem]{Definition}
\newtheorem{lemma}[theorem]{Lemma}
\newtheorem{prop}[theorem]{Proposition}
\newtheorem{cor}[theorem]{Corollary}
\newtheorem{rem}[theorem]{Remark}
\newtheorem{ex}[theorem]{Example}

\newcommand{\ad}{{\operatorname{ad}}}
\newcommand{\mm}{{\operatorname{M}}}
\newcommand{\la}{\langle}
\newcommand{\ra}{\rangle}
\newcommand{\ip}{\la \cdot, \cdot \ra}
\newcommand{\tr}{\operatorname{tr}}
\newcommand{\pg}{\mathfrak{p}}
\newcommand{\kg}{\mathfrak{k}}
\newcommand{\g}{\mathfrak{g}}

\newcommand\C{{\mathbb C}}
\newcommand\R{{\mathbb R}}

\newcommand{\N}{\nabla}

\title[HCF on  Lie Groups and static metrics]{Hermitian Curvature flow on unimodular Lie groups and static invariant metrics}
\thanks{This work was supported by G.N.S.A.G.A. of I.N.d.A.M.}

\author{Ramiro A.~Lafuente}
\address{School of Mathematics and Physics, The University of Queensland, St Lucia QLD 4072, Australia}
\email{r.lafuente@uq.edu.au}

\author{Mattia Pujia}
\address{Dipartimento di Matematica G. Peano, Universit\`a di Torino, Via Carlo Alberto 10, 10123 Torino, Italy}
\email{mattia.pujia@unito.it}

\author{Luigi Vezzoni}
\address{Dipartimento di Matematica G. Peano, Universit\`a di Torino, Via Carlo Alberto 10, 10123 Torino, Italy}
\email{luigi.vezzoni@unito.it}

\date{\today}

\begin{document}

\begin{abstract}
We investigate the Hermitian curvature flow (HCF) of left-invariant metrics on complex unimodular Lie groups. We show that in this setting the flow is governed by the Ricci-flow type equation $\partial_tg_{t}=-{\rm Ric}^{1,1}
(g_t)$. The solution $g_t$ always exists for all positive times, and $(1 + t)^{-1}g_t$ converges as $t\to \infty$ in Cheeger-Gromov sense to a non-flat left-invariant soliton $(\bar G, \bar g)$. Moreover, up to homotheties on each of these groups there exists at most one left-invariant soliton solution, which is a static Hermitian metric if and only if the group is semisimple. In particular, compact quotients of  complex semisimple Lie groups yield examples of compact non-K\"ahler manifolds with static Hermitian metrics.  We also investigate the existence of static metrics on nilpotent Lie groups and we generalize a result in \cite{EFV} for the pluriclosed flow. In the last part of the paper we study the HCF on Lie groups with abelian complex structures. 
\end{abstract}

\maketitle

\section{Introduction} 
{\em Hermitian curvature flow} (or shortly HCF) is a natural parabolic flow of Hermitian metrics introduced in \cite{ST2} by Streets and Tian. It evolves an initial Hermitian metric in the direction of a Ricci-type tensor defined as a trace of the Chern curvature modified with some first order torsion terms. In the compact case, HCF is a gradient flow and has nice analytic properties, such as short time existence and stability near K\"ahler-Einstein metrics with non-positive Ricci curvature \cite{ST2}. For K\"ahler initial metrics, the flow reduces to the K\"ahler-Ricci flow.  
 
\medskip  
In order to describe the evolution equation more precisely, let $(M,g)$ be a Hermitian manifold of complex dimension $n$ with Chern connection $\N$ and let $\Omega$ be the curvature tensor of $\N$. Denote by $S$ the $(1,1)$-symmetric tensor given by
\[
	S_{i\bar j}=g^{k\bar l}\Omega_{k\bar li\bar j}\,. 
\]
Let $T$ be the torsion of $\nabla$, $T_{ij}^k$ its components and $Q=Q_{i\bar j}$ be  the $(1,1)$-tensor defined by
\[
	Q= \frac12Q^1 - \frac14Q^2- \frac12Q^3 + Q^4\,,
\] 
where the $Q^i$ are quadratic expressions on the torsion given by
\begin{align*}
&Q^1_{i\bar j}=g^{k\bar l}g^{m\bar n}T_{ik\bar n}T_{\bar j\bar l m}\,,\qquad  
Q^2_{i\bar j}=g^{\bar l k}g^{\bar n m}T_{\bar l\bar n i}T_{km \bar j}\,,\\
&Q^3_{i\bar j}=g^{\bar l k}g^{\bar n m}T_{ik\bar l}T_{\bar j\bar n m}\,,\qquad  
Q^4_{i\bar j}=\frac12g^{\bar l k}g^{\bar n m}(T_{mk\bar l}T_{\bar n\bar ji} +T_{\bar n\bar lk}T_{mi\bar j})\,,
\end{align*}
and  $T_{ij\bar k}:=g_{a \bar k }T_{ij}^a$. HCF is then defined as 
\begin{equation}\label{HCF}
\partial_tg_t=-K(g_t)\,,\qquad g_{|t=0}=g_0,
\end{equation}
where $g_0$ is a fixed initial Hermitian metric on $M$ and 
\begin{equation*}\label{eqn_K}
K(g):=S(g)-Q(g)\,. 
\end{equation*}


\medskip 
In the present paper we study the  behavior  and the existence of self-similar solutions to the HCF on Lie groups, with a special emphasis on complex unimodular Lie groups. Since uniqueness for solutions of \eqref{HCF} is not known in general beyond the compact case, we now specify what we mean by this. Let $(G,g_0)$ be a Lie group endowed with a left-invariant Hermitian metric. The tensor $K(g_0)$ is again left invariant, thus one  may  study the corresponding ODE on the Lie algebra $\mathfrak{g} = {\rm Lie}(G)$, for which short-time existence and uniqueness are of course guaranteed. This defines a solution $g_t$ to \eqref{HCF} consisting of left-invariant metrics, and throughout this article we will refer to it as `the' solution to the HCF.

In our first main result we completely describe the long-time behavior of the flow for left-invariant initial metrics on complex unimodular Lie groups.

\begin{theorem}\label{main_cxHCF}
For a complex unimodular Lie group $G$, the maximal solution $g_t$ to the HCF \eqref{HCF} starting at a left-invariant  Hermitian metric satisfies 
\[
	\tfrac{d}{dt}\,g_t=-{\rm Ric}^{1,1}(g_t)\,,
\]
where ${\rm Ric}(g_t)$ denotes the Levi-Civita Ricci tensor. The family of left-invariant Hermitian metrics $g_t$ is defined for all $t\in  (-\epsilon,\infty)$ for some $\epsilon>0$, and $(1 + t)^{-1}g_t$ converges as $t\to \infty$  to a non-flat left-invariant HCF soliton $(\bar G, \bar g)$, in the Cheeger-Gromov topology.  
\end{theorem}

By convergence in the Cheeger-Gromov topology we mean that for any increasing sequence of times there is a subsequence $(t_k)_{k\in \mathbb{N}}$ for which the corresponding Hermitian manifolds converge in the following sense: there exist biholomorphisms $\varphi_k : \Omega_k \subset \bar G \to \varphi_k(\Omega_k) \subset G $ taking the identity of $\bar G$ to the identity of $G$, such that the open sets $\Omega_k$ exhaust $\bar G$, and in addition $\varphi_k^* g_{t_k} \to \bar g$ as $k\to\infty$, in $C^\infty$ topology uniformly over compact subsets.

An {\em HCF soliton} is a Hermitian metric $g$ such that 
\begin{equation}\label{eqn_soliton}
	K(g)=\lambda g+\mathcal{L}_X g,
\end{equation}
for $\lambda\in \R$ and some complete holomorphic vector field $X$. Here $\mathcal L$ denotes the Lie derivative. The relevance of \eqref{eqn_soliton} is given by the fact that the corresponding HCF solution $g_t$ starting at $g$ is given by $g_t = c(t) \, \varphi_t^* g$, for some one-parameter family of biholomorphisms $\varphi_t : G \to G$ and some scalings $c(t) > 0$. We call a left-invariant HCF soliton \emph{algebraic} if $\varphi_t$ is a Lie group automorphism of $G$ for each $t$.

Particular examples of solitons are given by the so called  {\em static metrics}, that is, those satisfying the Einstein-type equation 
\begin{equation*}\label{eqn_static}
K=\lambda g, \qquad \lambda \in \mathbb{R}.
\end{equation*}
In the compact case, static metrics are critical points of the functional 
\[
	\mathbb F(g)= \int_M k\,dV_g
\]
acting on the space of Hermitian metrics, where $k=g^{\bar ji}K_{i\bar j}$. 

%

\medskip 

Our second main result is about the existence and uniqueness of left-invariant HCF solitons on complex unimodular Lie groups. It gives in particular a complete classification of left-invariant static metrics on such groups.

\begin{theorem}\label{main_cxsoliton}
A complex unimodular Lie group $G$ has at most one {\rm HCF} algebraic soliton up to homotheties. Moreover, $G$ has a static left-invariant metric if and only if it is semisimple, and in this case the `canonical metrics' (in the sense of Definition \ref{def_canonical}) induced by the Killing form of $\g$ are static with $\lambda < 0$.
\end{theorem}


In particular, this result yields the existence of many compact non-K\"ahler manifolds admitting static metrics.
Indeed, every semisimple Lie group admits a co-compact lattice and  compact quotients of complex Lie groups by lattices can be K\"ahler only if they are tori. 

A main ingredient in the proof of our main results is the fact that $K(g)$ can be viewed as a moment map for the action of the real reductive Lie group ${\rm Gl}(\g)$ on the vector space of Lie brackets $\Lambda^2 \g^* \otimes \g$ in the sense of real geometric invariant theory (real `GIT' for short), see Lemma \ref{M11} and Section \ref{sec_mm}. Recall that real GIT, initiated by Richardson and Slodowy in \cite{RS90} and further developed by \cite{standard,HS07,HSS08,EJ} among others, has been fruitfully applied in recent years to the study of problems involving geometric flows with symmetryies, see e.g.~\cite{lauret3,BL18,ALpluri}. We adopt in this article the setting and notation used in \cite{realGIT}, and refer the interested reader to that article and the references therein for further details.

The above is in turn based on the formula $K(g)={\rm Ric}^{1,1}(g)$,
 which holds for any left-invariant Hermitian metric $g$ on a complex unimodular Lie group $G$, and which we prove in Proposition \ref{prop_unimod}. 
 We also have that $K(g)={\rm Ric}(g)$ if and only if the Killing form of the Lie algebra $\g$ of $G$ vanishes. In particular, for a complex nilpotent Lie group with Hermitian left-invariant metric, the Ricci flow and the HCF coincide. It is interesting to note that, in this case, our results agree with those obtained by Lauret for the Ricci flow \cite{lauret3}.

Let us also mention that other flows in the HCF family have been studied quite extensively in the literature. By HCF family we mean the following: since short-time existence of HCF \eqref{HCF} in the compact case follows from the ellipticity of the operator 
\[
g \mapsto -S(g)
\] 
acting on the space of Hermitian metrics compatible with a fixed complex structure, it is clear that by changing the lower order torsion term $Q(g)$ in \eqref{HCF} one also gets a well-defined parabolic evolution equation.  
In this direction, in \cite{ST} Streets and Tian introduced the {\em pluriclosed flow} (shortly PCF), a modification of the HCF which preserves the {\em pluriclosed condition} $\partial \bar \partial \omega=0$, defined by   
\begin{equation}\label{PCF}
\partial_tg_t=-S(g_t)+Q^1(g_t)\,,\quad g_{|t=0}=g_0.
\end{equation} 
It evolves an initial pluriclosed form $\omega_0$ via $\partial_t\omega_t=-(\rho^B)^{1,1}$, where $\rho^B$ is the Ricci-form of the Bismut connection. Some  regularity results involving the PCF are known \cite{S2, ST, ST4}, the PCF preserves the  generalized K\"ahler condition, and is a powerful tool in generalized geometry \cite{apostolov, S, S2, ST3}. In the case of homogeneous Hermitian structures, the PCF on Lie groups  was initially studied in \cite{EFV2}, where it is proved that the flow on $2$-step nilpotent SKT Lie groups  has always a long-time solution (see also \cite{pujia}). The proof of the result was obtained by adapting the bracket flow trick introduced by Lauret in \cite{lauret3} to study the Ricci flow on Lie groups. In \cite{boling}, Boling studied locally homogeneous solutions to the PCF on compact complex surfaces, and in \cite{lauret} Lauret gives a regularity result for a class of flows on Lie groups which includes PCF. Also, an analogous result to Theorem \ref{main_cxHCF} for the PCF on $2$-step nilmanifolds was recently obtained in \cite{ALpluri}.

Recently, Ustinovskiy introduced in  \cite{yuri} the following modified HCF
\begin{equation}\label{yuriflow}
\partial_tg_t=-S(g_t)-\frac12Q^ 2(g_t)\,,\quad g_{|t=0}=g_0.
\end{equation}
This last equation preserves non-negativity of the holomorphic bisectional curvature of the Chern connection. In \cite{yuri2} Ustinovskiy studied the modified HCF on complex Lie groups and on complex homogeneous spaces and he proved that the flow preserves the assumption on a Hermitian metric on a complex homogenous spaces $G/H$ to be \emph{induced} by an invariant metric on $G$. Hence if $g_0$ is an induced metric on $G/H$ equation \eqref{yuriflow} is equivalent to an ODE on the Lie algebra of $G$. Note that induced metrics are not in general homogenous, i.e. they need not be invariant under a transitive group action.

Regarding the existence of static metrics for other flows in the HCF family, we recall that bi-invariant metrics on compact semisimple Lie groups compatible with a left-invariant complex structure are pluriclosed and static for the PCF, with $\lambda=0$. On the other hand, it is still an open problem to find an example of a compact, non-K\"ahler manifold admitting a pluriclosed metric which is static for the PCF but has $\lambda\neq 0$.  Such metrics do not exist on compact complex non-K\"ahler surfaces \cite{ST,TJLi}, on non-abelian nilpotent Lie groups, and on some special solvable Lie groups
 \cite{ EFV2,FK1,FK2}. On the other hand, let us mention here that solitons of the the pluriclosed flow on compact complex surfaces were classified in \cite{Ssolitons}. 

%
%

%
%

\medskip

The paper is organized as follows. In Section \ref{firstsection} we give general formulae for the tensors $S$ and $Q^i$ for left-invariant metrics on Lie groups, in terms of the  bracket's components. We also point out how $K$ is related to the Ricci tensor in case the Lie group is complex. In Section \ref{sec_mmm} we prove Theorems \ref{main_cxHCF} and \ref{main_cxsoliton}.  In Section 4 we study explicitly the HCF on $3$-dimensional complex Lie groups. In Section 5 we study the existence  of static metrics on nilpotent Lie groups, while in the last section we take into account the evolution of Hermitian metrics compatible with abelian complex structures.

\bigskip
\noindent {\bf Notation and conventions.} In the indicial expressions, the symbol of sum over repeated indices will be omitted. When talking about a \emph{complex Lie group}, it will always be understood that it is endowed with the complex structure that makes multiplication a holomorphic map (i.e. $J$ is bi-invariant).
 
\bigskip
\noindent {\bf Acknowledgments.} The research of the present paper was originated by some conversations between the third author and Jorge Lauret.  The authors are very grateful to him for many useful insights.  Moreover, the authors would like to thank also Marco Radeschi for a number of useful conversations, and Fabio Podest\`a for his interest on  the paper.

\section{Some general formulae on Lie groups}\label{firstsection}

\subsection{The tensor $K$ on Lie groups}


Let $G$ be a Lie group equipped with a left-invariant complex structure $J$. We denote by $\g$ the Lie algebra of $G$ and by $\mu$ its Lie bracket. In this section we compute the tensor $K$ of a left-invariant Hermitian metric $g$  on $G$ in terms of the components of $\mu$ with respect to a $g$-unitary basis.  

Let $\{Z_1,\dots,Z_n\}$ be a left-invariant $g$-{unitary} frame on $G$, and let $\nabla$ be the \emph{Chern connection}, i.e.~ the unique Hermitian connection ($\nabla J = \nabla g = 0$) for which the $(1,1)$-part of the torsion tensor vanishes. The latter property implies that we can write 
$$
\nabla_{\bar k}Z_l=\nabla_l Z_{\bar k}+\mu(Z_{\bar k},Z_l)\,,
$$
or, in terms of the Christoffel symbols of $\nabla$,
$$
\Gamma_{\bar kl}^r=\mu_{\bar kl}^r\,,\quad \Gamma_{k\bar l}^{\bar r}=\mu_{ k \bar l}^{\bar r}\,.
$$
Using $\nabla J=\nabla g=0$, we have 
$$
g(\nabla_{Z_k}Z_r,Z_{\bar j})=-g(Z_r,\nabla_{Z_k}Z_{\bar j})=-g(Z_r,\mu(Z_k,Z_{\bar j}))=-\mu_{k\bar j}^{\bar r}\,,
$$
i.e.
$$
\Gamma_{kr}^{j}=-\mu_{k\bar j}^{\bar r} \,.
$$
We have 
$$
\Omega_{k\bar li\bar j}=g(\nabla_{k}\nabla_{\bar l}Z_i,Z_{\bar j})-g(\nabla_{\bar l}\nabla_kZ_i,Z_{\bar j})-g(\nabla_{\mu(Z_k,Z_{\bar{l}})}Z_i,Z_{\bar j})\,,
$$
with 
$$
\begin{aligned}
g(\nabla_{k}\nabla_{\bar l}Z_i,Z_{\bar j})=\Gamma_{\bar l i}^r\Gamma_{kr}^{j}=-\mu_{\bar li}^r\mu_{k\bar j}^{\bar r}
\end{aligned}
$$
and
$$
g(\nabla_{\bar l}\nabla_kZ_i,Z_{\bar j})=\Gamma_{ki}^r \Gamma_{{\bar l}r}^j=-\mu_{k\bar r}^{\bar i} \mu_{\bar l r}^j
$$
and 
$$
\begin{aligned}
g(\nabla_{\mu(Z_k,Z_{\bar{l}})}Z_i,Z_{\bar j})=\,\mu_{k\bar l}^r\Gamma_{ri}^j+\mu_{k\bar l}^{\bar r}\Gamma_{\bar ri}^{j}
=-\mu_{k\bar l}^r\mu_{r\bar j}^{\bar i}+\mu_{k\bar l}^{\bar r}\mu_{\bar ri}^{j}\,.
\end{aligned}
$$
Therefore,
$$
\Omega_{k\bar li\bar j}=-\mu_{\bar li}^r\mu_{k\bar j}^{\bar r}+\mu_{k\bar r}^{\bar i} \mu_{\bar l r}^j+\mu_{k\bar l}^r\mu_{r\bar j}^{\bar i}-\mu_{k\bar l}^{\bar r}\mu_{\bar ri}^{j}
$$
and 
$$
S_{i\bar j}=-\mu_{\bar ki}^r\mu_{k\bar j}^{\bar r}+\mu_{k\bar r}^{\bar i} \mu_{\bar k r}^j+\mu_{k\bar k}^r\mu_{r\bar j}^{\bar i}-\mu_{k\bar k}^{\bar r}\mu_{\bar ri}^{j}\,.
$$
In particular the Chern scalar curvature $s=g^{\bar ji}S_{i\bar j}$  takes the following expression in terms of the components of $\mu$
$$
s=\mu_{k\bar k}^r\mu_{r\bar i}^{\bar i}-\mu_{k\bar k}^{\bar r}\mu_{\bar ri}^{i}\,.
$$

Since 
$$
T_{ij}:=\nabla_{i}Z_j-\nabla_jZ_i-\mu(Z_i,Z_j)\,,
$$
we have 
$$
T_{ij}^k:=\Gamma_{ij}^k-\Gamma_{ji}^k-\mu_{ij}^k
$$
and
$$
T_{ij}^k:=-\mu_{i\bar k}^{\bar j}+\mu_{j\bar k}^{\bar i}-\mu_{ij}^k\,,\quad 
T_{ij\bar m}=-\mu_{i\bar m}^{\bar j}+\mu_{j\bar m}^{\bar i}-\mu_{ij}^m\,.
$$
\newline
It follows that
$$
\begin{aligned}
Q^1_{i\bar j}=&\,T_{ik\bar r}T_{\bar j\bar k r}=\left(-\mu_{i\bar r}^{\bar k}+\mu_{k\bar r}^{\bar i}-\mu_{ik}^r\right)
\left(-\mu_{\bar j r}^{k}+\mu_{\bar kr}^{j}-\mu_{\bar j\bar k}^{\bar r}\right)\\
=&\,\mu_{i\bar r}^{\bar k}\mu_{\bar j r}^{k}-\mu_{i\bar r}^{\bar k}\mu_{\bar kr}^{j}+\mu_{i\bar r}^{\bar k}\mu_{\bar j\bar k}^{\bar r}-\mu_{k\bar r}^{\bar i}\mu_{\bar j r}^{k}+\mu_{k\bar r}^{\bar i}\mu_{\bar kr}^{j}-\mu_{k\bar r}^{\bar i}\mu_{\bar j\bar k}^{\bar r}+\mu_{ik}^r\mu_{\bar j r}^{k}-\mu_{ik}^r\mu_{\bar kr}^{j}+ \mu_{ik}^r\mu_{\bar j\bar k}^{\bar r} \,.
\end{aligned}
$$
\newline
In the same way 
$$
\begin{aligned}
Q^2_{i\bar j}=&T_{\bar k\bar r i}T_{kr \bar j}=\left(-\mu_{\bar ki}^{r}+\mu_{\bar ri}^{k}-\mu_{\bar k\bar r}^{\bar i}\right)
\left(-\mu_{k\bar j}^{\bar r}+\mu_{r\bar j}^{\bar k}-\mu_{kr}^j\right)\\
=&\mu_{\bar ki}^{r}\mu_{k\bar j}^{\bar r}-\mu_{\bar ki}^{r}\mu_{r\bar j}^{\bar k}+\mu_{\bar ki}^{r}\mu_{kr}^j-\mu_{\bar ri}^{k}\mu_{k\bar j}^{\bar r}+\mu_{\bar ri}^{k}\mu_{r\bar j}^{\bar k}-\mu_{\bar ri}^{k}\mu_{kr}^j+\mu_{\bar k\bar r}^{\bar i}\mu_{k\bar j}^{\bar r}-\mu_{\bar k\bar r}^{\bar i}\mu_{r\bar j}^{\bar k}+\mu_{\bar k\bar r}^{\bar i}\mu_{kr}^j
\end{aligned}
$$
\newline
and 
$$
\begin{aligned}
Q^3_{i\bar j}&=T_{ik\bar k}T_{\bar j\bar r r}=
\left(-\mu_{i\bar k}^{\bar k}+\mu_{k\bar k}^{\bar i}-\mu_{ik}^k\right)
\left(-\mu_{\bar jr}^r+\mu_{\bar r r}^j-\mu_{\bar j\bar r}^{\bar r}\right)\\
&=\mu_{i\bar k}^{\bar k}\mu_{\bar jr}^r-\mu_{i\bar k}^{\bar k}\mu_{\bar r r}^j+\mu_{i\bar k}^{\bar k}\mu_{\bar j\bar r}^{\bar r}
-\mu_{k\bar k}^{\bar i}\mu_{\bar jr}^r+\mu_{k\bar k}^{\bar i}\mu_{\bar r r}^j-\mu_{k\bar k}^{\bar i}\mu_{\bar j\bar r}^{\bar r}
+\mu_{ik}^k\mu_{\bar jr}^r-\mu_{ik}^k\mu_{\bar r r}^j+\mu_{ik}^k\mu_{\bar j\bar r}^{\bar r}
 \end{aligned}
$$
and  
$$
\begin{aligned}
2Q^4_{i\bar j}=& \mu_{r \bar k}^{\bar k}\mu_{\bar ri}^{j} - \mu_{r \bar k}^{\bar k} \mu_{\bar ji}^{r}+ \mu_{r \bar k}^{\bar k} \mu_{\bar r \bar j}^{\bar i}
-\mu_{k \bar k}^{\bar r}\mu_{\bar ri}^{j} + \mu_{k \bar k}^{\bar r} \mu_{\bar ji}^{r} - \mu_{k \bar k}^{\bar r} \mu_{\bar r \bar j}^{\bar i}
+\mu_{rk}^{k}\mu_{\bar ri}^{j} - \mu_{rk}^{k} \mu_{\bar ji}^{r} + \mu_{rk}^{k}\mu_{\bar r \bar j}^{\bar i} \\
+& \mu_{\bar r k}^{k} \mu_{r \bar j}^{\bar i}- \mu_{\bar r k}^{k} \mu_{i \bar j}^{\bar r}+ \mu_{\bar r k}^{k} \mu_{ri}^{j}
-\mu_{\bar kk}^{r} \mu_{r \bar j}^{\bar i} + \mu_{\bar kk}^{r} \mu_{i \bar j}^{\bar r} - \mu_{\bar kk}^{r} \mu_{ri}^{j}
+\mu_{\bar r \bar k}^{\bar k} \mu_{r \bar j}^{\bar i}- \mu_{\bar r \bar k}^{\bar k} \mu_{i \bar j}^{\bar r}+ \mu_{\bar r \bar k}^{\bar k}  \mu_{ri}^{j} \,.
\end{aligned}
$$

Notice that if $G$ is a complex Lie group, then the mixed brackets $\mu(Z_r,Z_{\bar s})$ vanish and so we have the following relations    
$$
S=0\,,\quad Q_{i\bar j}^1=\mu_{ik}^r\mu_{\bar j\bar k}^{\bar r} \,,\quad 
Q_{i\bar j}^2=\mu_{\bar k\bar r}^{\bar i} \mu_{kr }^j  \,,\quad 
Q^3_{i\bar j}=\mu_{ik}^k\mu_{\bar j\bar r}^{\bar r}\,,\quad 
Q^4_{i\bar j}=\frac{1}{2}\left(\mu_{rk}^k\mu_{\bar r \bar j}^{\bar i}+\mu_{ri}^j\mu_{\bar r\bar k}^{\bar k} \right)\,.
$$
In particular $K$ reduces to 
\begin{equation}\label{CX2}
K_{i\bar j}=-\frac12\mu_{ik}^r\mu_{\bar j\bar k}^{\bar r} +\frac14\mu_{\bar k\bar r}^{\bar i} \mu_{kr }^j+\frac12\mu_{ik}^k\mu_{\bar j\bar r}^{\bar r} - \frac{1}{2}\left(\mu_{rk}^k\mu_{\bar r \bar j}^{\bar i}+\mu_{ri}^j\mu_{\bar r\bar k}^{\bar k} \right).
\end{equation}
%
 
\medskip  
We now use the previous formulas to show how $K$ and the Riemannian Ricci tensor are related on a complex Lie group.  According to  \cite[(7.33)]{Besse}  the Riemannian Ricci tensor of a left-invariant metric $g$ on a Lie group $G$ can be written as
\[
	{\rm Ric}= {\rm M}-\tfrac 12\, {\rm B}-{\rm S}({\rm ad}_{H})\, ,
\]
where for $X,Y$ in $\mathfrak g$
\begin{equation}\label{Mmu}
{\rm M}(X,Y)=-\frac 12  \, g( \mu(X,X_k),\mu(Y,X_k))+\frac 14  \, g( \mu(X_k,X_j),X) g(\mu(X_k,X_j),Y)\,,
\end{equation}
$\{X_r\}$ being an orthonormal basis;  
${\rm B}(X,Y) = {\rm tr} (\ad_X \ad_Y)$ 
is the Killing form of $\g$, $H$ is the {\em mean curvature vector}, uniquely determined by the relation 
$g(H,X)={\rm tr}\,{\rm ad}_X$, for all $X \in \g$, and 
\[
	{\rm S}({\rm ad}_{H})(X,Y)=\frac12\big(g(\mu(H,X),Y)+g(\mu(H,Y),X)\big)\,. 
\]

\begin{lemma}\label{M11}
Let $G$ be a complex Lie group equipped with a left-invariant  Hermitian metric $g$. Then, ${\rm M}$ and $S({\rm ad}_H)$ are of type $(1,1)$, while ${\rm B}$ is of type $(2,0)+(0,2)$. In particular,
\begin{equation}\label{1}
{\rm Ric}^{1,1}={\rm M}-{\rm S}({\rm ad}_{H})\,,\quad {\rm Ric}^{2,0+0,2}=-\frac12 {\rm B}.
\end{equation}
Moreover, 
\begin{equation}\label{Knonuni}
K={\rm Ric}^{1,1}+\frac12 Q^3\,.
\end{equation}
\end{lemma}

\begin{proof}
Let $\{X_1,\dots,X_{2n}\}$ be a $J$-invariant orthonormal basis of $\g$, where $J$ is the complex structure of $G$. We directly compute 
$$
\begin{aligned}
{\rm M}(JX,JY)=&\,-\frac 12g(\mu(JX,X_k),\mu(JY,X_k))+\frac 14 g(\mu(X_k,X_j),JX) g(\mu(X_k,X_j),JY)\\
=&\,-\frac 12 g(J\mu(X,X_k),J\mu(Y,X_k))+\frac 14 g(\mu(JX_k,X_j),X) g(\mu(JX_k,X_j),Y)\\
=&\,-\frac 12 g(\mu(X,X_k),\mu(Y,X_k))+\frac 14 g(\mu(X_k,X_j),X) g(\mu(X_k,X_j),Y)\\
=&\,{\rm M}(X,Y)\,,
\end{aligned}
$$
for every $X,Y$ in $\g$, which implies that ${\rm M}$ is of type $(1,1)$. Moreover, 
$$
\begin{aligned}
{\rm S}({\rm ad}_{H})(JX,JY)=&\,\frac12\left(g(\mu(H,JX),JY)+g(\mu(H,JY),JX)\right)\\
=&\,\frac12\left(g(J\mu(H,X),JY)+g(J\mu(H,Y),JX)\right)={\rm S}({\rm ad}_{H})(X,Y)\,,
\end{aligned}
$$
and 
$$
{\rm B}(JX,JY)={\rm tr}(\ad_{JX}\ad_{JY})={\rm tr}(J^2\ad_{X}\ad_{Y})=-{\rm B}(X,Y)
$$
which imply \eqref{1}. 

Let  $\{Z_r\}$ be a unitary frame. We have 
$$
{\rm M}(Z_i,Z_{\bar j})=-\frac12 \mu_{ik}^{ r}\mu_{\bar j\bar k}^{\bar r}+\frac14\mu_{\bar k\bar r}^{\bar i} \mu_{kr }^j\\,
$$
and 
$$
\begin{aligned}
{\rm S}({\rm ad}_{H})(Z_i,Z_{\bar j})&=\,\frac12\left(g(\mu(H,Z_i),Z_{\bar j})+g(\mu(H,Z_{\bar j}),Z_{i})\right)\\
&=\, \frac12\left(H_k 	\mu_{ki}^j+H_{\bar k} \mu_{\bar k\bar j}^{\bar i}\right)\,.
\end{aligned}
$$
Since 
$$
H_{k}=g(H,Z_{\bar k})={\rm tr}\,{\rm ad}_{Z_{\bar k}}=\mu_{\bar k\bar l}^{\bar l}
$$
we infer 
$$
{\rm S}({\rm ad}_{H})(Z_i,Z_{\bar j})=\frac12\left(\mu_{\bar k\bar l}^{\bar l} \mu_{ki}^j+\mu_{ kl}^{ l} \mu_{\bar k\bar j}^{\bar i}\right)
$$
and \eqref{CX2} implies the second part of the statement. 
\end{proof}

 Lemma \ref{M11} has the following direct consequence:
\begin{cor}\label{20+02} 
Let $G$ be a complex semisimple Lie group. Then, the Ricci tensor  of a left-invariant Hermitian metric on $G$ is never of type $(1,1)$. In particular $G$ has no  left-invariant Hermitian metrics which are also Einstein. 
\end{cor}

Next we focus on {\em unimodular} complex Lie groups. We recall that a Lie group $G$ is unimodular if ${\rm tr}\, {\rm ad}_X=0$ for all $X\in \g$. If $G$ is equipped with a left-invariant Hermitian structure, the unimodular condition reads in terms of a left-invariant unitary frame as 
$$
\mu_{ir}^{ r}+\mu_{i\bar r}^{\bar r}=0\,,\quad i=1,\dots,n\,.   
$$ 
If $G$ is a complex Lie group, the  unimodular condition simply reduces to
\begin{equation}\label{unimodular}
\mu_{ir}^{ r}=0\,,\quad i=1,\dots,n\,.
\end{equation}

\begin{prop}\label{prop_unimod}
Let $G$ be a complex Lie group with a left-invariant Hermitian metric $g$ and denote by $r$ the Riemannian scalar curvature of $g$. The following facts are equivalent: 
\begin{enumerate}[label=(\roman*)]
\item[$1.$] $r=k$;

\vspace{0.1cm}
\item[$2.$] $G$ is unimodular;

\vspace{0.1cm}
\item[$3.$] $K={\rm Ric}^{1,1}$.
\end{enumerate}
Moreover, if one of these holds, then $K={\rm Ric}$ if and only if the Killing form of $\g$ vanishes. In particular, if $G$ is nilpotent then $K={\rm Ric}$.
\end{prop}

\begin{proof}

We have  
$$
k={\rm tr}_{g}\, K={\rm tr}_{g}\, {\rm Ric}^{1,1}+\frac12 {\rm tr}_g\,Q^{3}=r+\frac12 {\rm tr}_g\,Q^{3}\,.
$$
Since  
$$
 {\rm tr}_g\,Q^{3}=\mu_{ik}^k\mu_{\bar i\bar r}^{\bar r}\,,
$$
we have 
$$
Q^3=0\iff {\rm tr}_g\,Q^{3}=0
$$
and  the equivalences $1\iff 2\iff 3$ follow.  

Finally, it is well-known that a nilpotent Lie algebra is unimodular and has zero Killing form; thus, the last part of the statement follows.
\end{proof}
Note that if $g$ is a left-invariant  Hermitian metric on a complex unimodular Lie group, by \eqref{CX2} and Lemma \ref{M11}, with respect to a unitary frame $K$ is given by 
\begin{equation}\label{CX}
K_{i\bar j}=-\frac12\mu_{ik}^r\mu_{\bar j\bar k}^{\bar r} +\frac14\mu_{\bar k\bar r}^{\bar i} \mu_{kr }^j.
\end{equation}

The  following corollary will have a central in role in the study on the HCF on unimodular complex Lie groups 
\begin{cor}
\label{M-flow}
Let $G$ be a complex unimodular Lie group with a left-invariant Hermitian metric $g_0$. Then, the {\em HCF} starting at $g_0$ can be written simply as 
\[
	\partial_tg_t=-{\rm M} (g_t)\,,\quad g_{|t=0}=g_0\,,
\]
where ${\rm M}$ is the left-invariant tensor defined via \eqref{Mmu}.  
\end{cor} 

\begin{proof}
This follows immediately from Proposition \ref{prop_unimod} and Lemma \ref{M11}, using the fact that for a unimodular Lie group one has $H = 0$.
\end{proof}



\subsection{The bracket flow approach}\label{sec_bf}
In this section we recall  the bracket flow device introduced by Lauret in \cite{lauret3} and used in \cite{ALpluri,EFV2,lauret,lauret1,lauretval} to study geometric flows of (almost)-Hermitian structures. The trick consists in regarding the flow as an evolution equation for Lie brackets instead of metrics.  The argument can be used to study evolutions of a large class of geometric structures on homogeneous spaces \cite{lauret1}. 

Let $G$ be a Lie group with Lie algebra $\g$. The Lie  bracket $\mu_0$ of $\g$ is an element of the \emph{variety of Lie algebras}
\[
\mathcal C=\left\{\mu\in\Lambda^2\g^*\otimes \g \mbox{ satisfying the Jacobi identity }\right\}\subseteq \Lambda^2\g^*\otimes \g\,. 
\] 
Let  $(g_t)_{t\in I}$, $0\in I\subset \R$, be the unique solution to the following  flow on $G$:
\begin{equation}\label{genflow}
	\partial_tg_t= -P(g_t)\,,\qquad g_{|t=0}=g_0.
\end{equation}
Here, the map $P$ is any `geometric' operator (that is, invariant under diffeomorphisms) from the space of Riemannian metrics on $G$ to the space of symmetric $2$-tensors (for instance, $P(g)=2\, {\rm Ric}(g)$ in the case of the Ricci flow). Diffeomorphism invariance implies that $P(g)$ is left-invariant for any left-invariant $g$. Thus, as explained in the discussion before Theorem \ref{main_cxHCF}, for a given left-invariant initial metric $g_0$ there exists a unique solution $g_t$ to \eqref{genflow} consisting entirely of left-invariant metrics. Evaluating everything at the identity $e\in G$ and using $T_e G \simeq \mathfrak{g}$ one can write $g_t$ in terms of the starting metric $g_0$ as 
\[
	g_t(\cdot,\cdot)=g_0(A_t\cdot,A_t\cdot),
\]
for a smooth curve $(A_t)_{t\in I} \subset {\rm Gl}(\g)$ with $A_0 = {\rm Id}_\g$. Even though the curve $A_t$ is of course not unique, Theorem 1.1 in \cite{lauret} shows that one can choose it so that the corresponding family of brackets
\[
	\mu_t := A_t \cdot \mu_0
\]
satisfies the  {\em bracket flow} equation  
\begin{equation}\label{bracketflow}
	\frac{d}{dt}\mu_t = -  \pi \big(P_{\mu_t} \big) \mu_t, \qquad \mu_{t|=0}=\mu_0\,.
\end{equation}
Here, $(A, \mu) \mapsto A\cdot \mu$ is the `change of basis' linear action of ${\rm Gl}(\g)$ on $\mathcal C \subset \Lambda^2 \g^* \otimes \g$, given by 
\[
	A\cdot \mu(\cdot,\cdot)=A\,\mu(A^{-1}\cdot,A^{-1}\cdot)\,,
\]
and $\pi : {\rm End}(\g) \to {\rm End}( \Lambda^2 \g^* \otimes \g)$ is the corresponding Lie algebra representation
\begin{equation}\label{delta}
 \big(\pi(E)\mu \big) (X,Y) :=  E(\mu(X,Y)) - \mu(E(X),Y) - \mu(X,E (Y))\,, \quad X, Y\in \g, \quad E\in {\rm End}(\g).
\end{equation}
Finally, $P_{\mu_t}\in {\rm End}(\g)$ is related to the value of $P(g_t)$ at $e\in G$ by 
\[
	P_{ \mu_t}=A_tP_{g_t}A_t^{-1}\,,\qquad g_t(P_{g_t}\cdot,\cdot)=P(g_t)(\cdot,\cdot)\,.
\]

The bracket flow is a powerful tool for proving results about the long-time existence and regularity of geometric flows on homogeneous spaces. For instance, in certain cases (such as the Ricci flow and the pluriclosed flow on $2$-step nilpotent Lie groups), one has
\[
	\tfrac{d}{dt}\langle \mu_t,\mu_t\rangle\leq 0, 
\]
which implies long-time existence for the ODE \eqref{bracketflow} and hence also for the original flow.

\section{The $\mm$-flow}\label{sec_mmm}

In view of Corollary \ref{M-flow}, the HCF  starting at a left-invariant metric $g_0$ on a complex unimodular  Lie group $G$ reduces to the ODE 
\begin{equation}\label{m-flow}
\tfrac{d}{dt}g_t=-{\rm M} (g_t)\,,\quad g_{|t=0}=g_0\,,
\end{equation}
which we call the {\em $\mm$-flow}. The same is true for the Ricci flow of left-invariant metrics on a nilpotent Lie group \cite{lauret3}, and so \eqref{m-flow} models both the HCF on complex Lie groups and the Ricci flow on simply-connected nilpotent Lie groups. 

In this section we study  \eqref{m-flow}  in the \lq\lq general\rq\rq\, setting of a Lie group $G$ equipped with a left-invariant initial metric $g_0$. Theorems \ref{main_cxHCF} and \ref{main_cxsoliton} arise from the results of this section.

\subsection{The tensor $\mm$ is a moment map}\label{sec_mm}
After fixing a backgroung metric $g_0$ on $G$, we have an induced scalar product on $\g$, and this in turn gives naturally defined scalar products on any tensor product of $\g$ and its duals. In particular, on ${\rm End}(\g)$ we have the associated scalar product 
\[
	\la  A, B \ra := {\rm tr}\, A B^t,  \qquad A, B \in {\rm End}(\g),
\]
where the transpose is taken with respect to $g_0$. If $\{ e_i\}$ is a $g_0$-orthonormal basis of $\g$ and $\{e^i \}$ its dual, then on $\Lambda^2 \g^* \otimes \g$ one may consider the scalar product, also denoted by $\ip$, making $\{e^i\wedge e^j \otimes e_k \}_{i<j,k}$ orthonormal. It is easy to see that this does not depend on the choice of basis.

Given another left-invariant metric $g(\cdot, \cdot) = g_0(A \cdot, A \cdot)$ on $G$, $A\in {\rm Gl}(\g)$, we follow the approach described in Section \ref{sec_bf} and set $\mu := A \cdot \mu_0$ and 
\[
	\mm_\mu := A \mm_g A^{-1}, \qquad 	g(\mm_g \cdot, \cdot) = \mm(g)(\cdot, \cdot).
\] 
Notice that $\mm_\mu$ is $g_0$-symmetric.

Having introduced the ${\rm Gl}(\g)$-action on $\Lambda^2 \g^* \otimes \g$ and the above notation, we are now in a position to state a remarkable property of the tensor $\mm$, first observed in \cite{lauretJFA} for the case of the corresponding complexified representation of ${\rm Gl}_n(\C)$ and proved in \cite[Prop.3.5]{lauret06} in the real setting. For convenience of the reader we include a short proof in our notation. We refer to \cite{realGIT}  and the references therein for further details on moment maps in this setting.

\begin{prop}\cite{lauret06}\label{prop_mm}
The mapping 
\[
	\Lambda^2 \g^* \otimes \g \setminus \{ 0\} \to {\rm End}(\g), \qquad\mu \mapsto \frac{4}{\Vert \mu \Vert^2} \, \mm_\mu,
\] 
is a moment map for the linear ${\rm Gl}(\g)$-action on $\Lambda^2 \g^* \otimes \g$, in the sense of real geometric invariant theory. That is, 
\begin{equation}\label{eqn_mm}
	\la \mm_\mu, E \ra = \tfrac14 \, \la  \pi(E)\mu, \mu \ra, \qquad \mbox{for all } E\in {\rm End}(\g), \quad \mu \in \Lambda^2 \g^* \otimes \g \setminus \{ 0\}\,,
\end{equation}
where $\pi(E)\mu$ is defined as in \eqref{delta}. 
\end{prop}

\begin{proof}
Let $\{ e_r\}$ be a $g_0$-orthonormal basis for $\g$. Since the statement to prove is linear in $E$, we may assume without loss of generality that the matrix of $E$ with respect to the basis $\{ e_r\}$ has a $1$ in its $i,j$ entry, and $0$'s otherwise. Denote by $\mu_{rs}^k$ the structure coefficients of a bracket $\mu$ with respect to $\{e_r \}$. Then, the left-hand side equals
\begin{align*}
	g_0(\mm_\mu e_j, e_i) = g(A^{-1} A \mm_g A^{-1} e_j, A^{-1} e_i) = g(\mm_g \tilde e_j, \tilde e_i) = \mm(g)(\tilde e_i, \tilde e_j),
\end{align*}
where $\tilde e_r := A^{-1} e_r$ is a $g$-orthonormal basis for $\g$. By \eqref{Mmu}, using $X_r = \tilde e_r$, we have
\begin{align*}
	\mm(g)(\tilde e_i, \tilde e_j) &= -\frac12 g(\mu_0(\tilde e_i, \tilde e_r), \mu_0 (\tilde e_j, \tilde e_r) ) + \frac14 g(\mu_0(\tilde e_r, \tilde e_s), \tilde e_i ) \, g(\mu_0(\tilde e_r, \tilde e_s), \tilde e_j) \\
	&= -\frac12 \, g_0 (\mu(e_i, e_r), \mu(e_j, e_r)) + \frac14 \, g_0 (\mu(e_r, e_s), e_i) \, g_0(\mu(e_r, e_s), e_j) \\
	&=  -\frac12 \, \mu_{ir}^k \mu_{jr}^k  + \frac14 \, \mu_{rs}^i \mu_{rs}^j.
\end{align*}

On the other hand, the right-hand side equals
\begin{align*}
	\frac14 \, g_0\big( (\pi(E)\mu)(e_r,e_s), e_k  \big) \, \, g_0\big( \mu(e_r,e_s), e_k  \big)   =&  \frac14  \, g_0\big( E\mu(e_r,e_s), e_k  \big) \, \, g_0\big( \mu(e_r,e_s), e_k  \big)   \\
	& - \frac14  \, g_0\big( \mu(Ee_r,e_s), e_k  \big) \, \, g_0\big( \mu(e_r,e_s), e_k  \big)   \\
	& - \frac14  \, g_0\big( \mu(e_r,Ee_s), e_k  \big) \, \, g_0\big( \mu(e_r,e_s), e_k  \big)   \\ 
	=& \frac14 \, \Big( \mu_{rs}^j \mu_{rs}^i -   \mu_{is}^k \mu_{js}^k -  \mu_{ri}^k \mu_{rj}^k \Big),
\end{align*}
which, by the skew-symmetry of $\mu$, coincides with the formula for $\la \mm_\mu, E \ra $ obtained above.
\end{proof}

%
%
%
%
%
%
%

\subsection{Long-time behaviour of the $\mm$-flow}

\begin{definition}
A left-invariant metric $g$ on a Lie group $G$ with Lie algebra $\g$ is an \emph{algebraic $\mm$-soliton} if its $\mm$ tensor satisfies
\begin{equation}\label{eqn_algmsol}
	\mm(g) = \lambda\, g + g( D \cdot, \cdot) + g(\cdot, D \cdot), \qquad \lambda \in \mathbb{R}, \quad D \in {\rm Der}(\g).
\end{equation}
\end{definition}

On a simply-connected Lie group $G$, algebraic $\mm$-solitons give rise to very special solutions to the $\mm$-flow \eqref{m-flow}. Indeed, let $\varphi_t \in {\rm Aut}(G)$ be the unique automorphism such that $d \varphi_t |_e = e^{-t D} \in {\rm Aut}(\g)$. Then a quick computation shows that a solution to the $\mm$-flow \eqref{m-flow} starting at $g$ is given by $g_t := (- \lambda t + 1) \, \varphi_t^* g$, and therefore $g$ is a soliton solution in the usual sense (it only evolves by scaling and pull-back by diffeomorphisms). 

\begin{theorem}\label{thm_mmflow}
Let $G$ be a Lie group with Lie algebra $\g$. Then, for any initial left-invariant metric $g_0$ the solution to 
\[
	\tfrac{d}{dt} g_t=-{\rm M}(g_t)\,,\qquad g_{|t=0}=g_0 \, ,
\] 
exists for all $t\in [0,\infty)$, and the rescaled metrics $(1 + t)^{-1}g_t $
converge as $t\to \infty$ in Cheeger-Gromov sense to a non-flat algebraic $\mm$-soliton $(\bar G, \bar g)$.
\end{theorem}

\begin{proof}
Let $(\mu_t)_{t\in I}$, $I\subset \R$, denote the corresponding maximal bracket flow solution satisfying \eqref{bracketflow}, where $P_{\mu_t} = \mm_{\mu_t}$. By the equivalence of the bracket flow and the original flow, 
 it suffices to prove that $\mu_t$ is defined for all $t\in [0,\infty)$. By looking at how the norm of $\mu_t$ evolves, we see that
\[
	\tfrac{d}{dt} \Vert \mu \Vert^2 = 2 \, \la \tfrac{d}{dt} \mu, \mu \ra = -2 \, \la \pi(\mm_\mu) \mu, \mu\ra = - 8 \, \Vert \mm_\mu \Vert^2 \leq 0,
\]
where in the last equality we used Proposition \ref{prop_mm}. It follows by standard ODE results that the solution $\mu_t$ is defined for all positive times. 

The proof of the last part of the statement will follow from three claims. The first one is that the norm-normalized bracket flow $\mu_t / \Vert \mu_t\Vert$ converges to a soliton bracket $\bar \mu$. The second one is that $\Vert \mu_t \Vert \sim t^{-1/2}$, thus up to a constant the metrics corresponding to the normalized brackets $\mu_t / \Vert \mu_t\Vert$ are asymptotic to the family $(1 + t)^{-1}g_t$ (recall that scaling the metric by a factor of $c>0$ is equivalent to scaling the corresponding bracket by $c^{-1/2}$ \cite[$\S$2.1]{Lau13}). The third claim is that convergence of the brackets yields subconvergence in the Cheeger-Gromov sense for the corresponding family of left-invariant metrics. 

In order  to prove the first claim, we recall that by \cite[Lemma~2.5]{ALpluri}, after a time reparameterization, the normalized solution $\nu_t := \mu_t / \Vert \mu_t \Vert$ solves the so called \emph{normalized bracket flow equation} 
\begin{equation}\label{eqn_normbf}
	\tfrac{d}{dt} \nu = -\pi(\mm_\nu + r_\nu \, {\rm Id}_\g) \nu,  
\end{equation}
where $r_\nu := \la \pi(\mm_\nu) \nu, \nu\ra  = 4 \, \Vert \mm_\nu \Vert^2$, and in the last equality we used Proposition \ref{prop_mm}. By \cite[Lemma~7.2]{realGIT} this last flow is (up to a constant time rescaling) the negative gradient flow of the real-analytic functional
\[
	F : \Lambda^2 \g^* \otimes \g \setminus \{ 0\} \to \R, \qquad \nu \mapsto \frac{\Vert \mm_\nu \Vert^2}{\Vert \nu \Vert^4},
\]
see also \cite[Corollary~3.5]{ALpluri} and \cite{lauret3}. By compactness, the family of unit norm brackets $(\nu_t)_{t\in [0,\infty)}$ must have an accumulation point $\bar \nu$. Now {\L}ojasiewicz's theorem on real-analytic gradient flows \cite{Loj63} implies that $\nu(t) \to \bar \nu$ as $t\to\infty$, and in particular $\bar \nu$ is a fixed point of \eqref{eqn_normbf}, that is
\[
\pi(\mm_{\bar\nu} + r_{\bar\nu} \, {\rm Id}_\g)\bar\nu=0.
\]
This directly implies that the corresponding metric is an algebraic $\mm$-soliton.

The second claim is proved in the second paragraph of the proof of \cite[Theorem~A]{ALpluri}. Finally, 
the last claim is a consequence of \cite[Corollary~6.20]{lauretconv}, and the theorem follows. 
\end{proof}

The proof of Theorem \ref{thm_mmflow} shows that the derivation in any algebraic $\mm$-soliton arising as a limit is symmetric, so that $\mm(g)$ satisfies the following simplified equation 
\begin{equation}\label{eqn_algmsolsym}
	\mm(g) = \lambda \, g  + g(D \cdot, \cdot), \qquad D\in  {\rm Der} (\g).	
\end{equation}
In fact, it can be proved directly that \eqref{eqn_algmsol} implies \eqref{eqn_algmsolsym}:



\begin{prop}\label{semiimplies}
On a non-abelian Lie group $G$, every algebraic $\mm$-soliton is expanding ($\lambda < 0$) and satisfies \eqref{eqn_algmsolsym}. 
\end{prop}{}

\begin{proof}
It is enough to show that if a left-invariant metric $g$ satisfies \eqref{eqn_algmsol}
then $D^t$ is also a derivation, where the transpose $(\cdot)^t$ is with respect to $g$. In the notation from  Section \ref{sec_mm}, if one chooses $g$ as background metric, then $A = {\rm Id}$ and Proposition \ref{prop_mm} implies 
\begin{equation}\label{eqn_mmgpi}
	\tr \mm_g E = \frac14 \, \la \pi(E)\mu, \mu \ra, \qquad \forall E\in {\rm End}(\g).
\end{equation}
By setting $E = [D, D^t]$, and using that $\pi$ is a Lie algebra homomorphism, that $\pi(E^t) = \pi(E)^t$, and that $\pi(D)\mu = 0$ (which by definition is equivalent to $D$ being a derivation), we obtain
\begin{align*}
	4 \, \tr \mm_g [D,D^t ] &=  \la \pi\big([D,D^t] \big) \mu, \mu \ra = \la \big[ \pi(D), \pi(D^t) \big] \mu, \mu \ra \\
	&= \la  \pi(D) \pi(D)^t \mu, \mu \ra - \la \pi(D)^t \pi(D) \mu,\mu \ra = \big\Vert  \pi(D^t)\mu  \big\Vert^2 
\end{align*}

On the other hand, \eqref{eqn_algmsol} is equivalent to 
\begin{equation}\label{eqn_mmsemiop}
	\mm_g = \lambda \, {\rm Id}_\g  +   D +  D^t .
\end{equation}
Putting this together with the above equation yields
\[
	\big\Vert  \pi(D^t)\mu  \big\Vert^2 
	= 4 \,  \lambda \tr [D,D^t] + 4 \tr D [D,D^t] + 4 \tr D^t [D,D^t] = 0,
\]
which implies that $D^t$ is a derivation of $\g$, as required.

Finally, to see that $\lambda < 0$  we first assume  that $D = 0$. In such a case \eqref{eqn_algmsolsym} and \eqref{eqn_mmgpi} yield
\[
	n \, \lambda = \tr \mm_g  = \frac14 \langle \pi({\rm Id} ) \mu, \mu \rangle = -\frac14 \Vert \mu \Vert^2 < 0,
\]
since $\pi({\rm Id}) \mu = - \mu$. If on the contrary $D \neq 0$, then
\begin{equation}\label{eqn_trD}
	\lambda \, \tr D  + \tr D^2 = \tr \mm_g D  = \frac14 \langle \pi(D) \mu,\mu \rangle = 0.
\end{equation}
Using that $\tr D^2 = \tr D D^t > 0$, the claim will follow once we show that $\tr D > 0$. To that end, notice that by tracing \eqref{eqn_algmsolsym} we obtain 
\[
	\lambda = -\frac1{n} \Big( \frac14 \Vert\mu\Vert^2 + \tr D \Big),
\]
and substituting this into \eqref{eqn_trD} yields
\[
	 \tr D^2 - \frac1{n} (\tr D)^2 = \frac1{4n} \Vert \mu \Vert^2 \tr D.
\]
The left-hand-side is non-negative by Cauchy-Schwarz, with equality if and only if $D = d \, {\rm Id}$, $d\neq 0$. Since $D$ is a derivation and $\g$ is non-abelian, we cannot have equality, hence $\tr D > 0$ as desired. 
\end{proof}

\subsection{Uniqueness and static metrics}

In order to characterize the groups which admit algebraic $\mm$-soliton metrics we will need the following lemma due to Dotti: 

\begin{lemma}\cite{Dotti88}\label{lem_dotti}
Let $G$ be a Lie group with Lie algebra $(\g, \mu)$, left-invariant metric $g$, and consider an abelian ideal $\mathfrak{i} \subset \g$.
  Then, 
\[
	\tr_g \mm(g)|_{\mathfrak{i} \times \mathfrak{i}} \geq 0.
\]
In particular, if $G$ admits a left-invariant metric $g$ with $\mm(g) < 0$, then $G$ is semisimple.
\end{lemma}

\begin{proof}
Let $\{Z_i \}$ be an orthonormal bases of $\mathfrak{i}$ and extend it to an orthonormal basis $\{ Z_i\} \cup \{ Y_j\}$ of $\g$. Since $\mu(\g,\mathfrak{i}) \subset \mathfrak{i}$, $\mu(\mathfrak{i}, \mathfrak{i}) = 0$,  formula \eqref{Mmu} for $X = Y = Z \in \mathfrak{i}$ can be rewritten as
\begin{align*}
	\mm(Z, Z) &= -\frac12 g(\mu(Z,Y_j), Z_i)g(\mu(Z,Y_j), Z_i)  + \frac12 g(\mu(Z_i,Y_j), Z) g(\mu(Z_i,Y_j), Z)  \\
		& \quad + \frac14 g(\mu(Y_j, Y_k), Z) g(\mu(Y_j, Y_k), Z).
\end{align*}
Summing as $Z$ ranges through the basis $\{ Z_i\}$ we get
\[
	\tr_g \mm(g)|_{\mathfrak{i} \times \mathfrak{i}} = \mm(Z_i,Z_i) = \frac14 g(\mu(Y_j,Y_k), Z_i)g(\mu(Y_j,Y_k), Z_i)  \geq 0.
\]
Finally, the last claim follows from the fact that a Lie algebra is semisimple if and only if it has no abelian ideals. 
\end{proof}

Now we introduce the notion of `canonical metric' of a semisimple Lie algebra.  Any semisimple Lie algebra $\g$ admits a \emph{Cartan decomposition} $\g = \mathfrak{k} \oplus \mathfrak{p}$, i.e.\ the following bracket relations are satisfied
\[
	[\kg,\kg] \subset \kg, \quad [\kg,\pg] \subset \pg, \qquad [\pg,\pg] \subset \kg,
\] 
and in addition the Killing form ${\rm B}$ of $\g$ is negative definite on $\kg$, positive definite on $\pg$ and ${\rm B}(\kg,\pg) = 0$. By switching the sign of ${\rm B}$ on $\kg$ we thus obtain an inner product on $\g$.

\begin{definition}\label{def_canonical}
A left-invariant metric on a semisimple Lie group $G$ with Lie algebra $\g$ is a \emph{canonical metric} if it induces on $\g$ the above defined inner product.
\end{definition}
The construction described above depends of course on the choice of Cartan decomposition, but since any two Cartan decompositions differ only by an automorphism (see e.g. \cite{Knapp}), any two canonical metrics $\ip$ and $\ip'$ on $\g$ are related by
\[
	\ip = \langle \varphi  \, \cdot, \varphi \, \cdot \rangle', \qquad \varphi \in {\rm Aut}(\g).
\]
Recall that the left-invariant metrics induced on $G$ by two such inner products on $\g$ are isometric, hence the canonical metric on a semisimple Lie group is unique up to isometry.

In the particular case of a complex semisimple Lie algebra  $\g$, a Cartan decomposition is obtained by considering a compact real form $\g_\R$ and setting $\kg = \g_\R$, $\pg = i \g_\R$, see \cite[Thm. 6.11]{Knapp}. Recall that $\g_\R$ is a real form of $\g$ if
\[
	\g = \g_\R \oplus i \, \g_\R
\]
and the Lie bracket of $\g$ is the $\C$-linear extension of the Lie bracket of $\g_\R$. The compact real Lie algebra $\g_\R$ is also semisimple and its Killing form ${\rm B}_{\g_\R}$ is negative definite. Clearly, the Killing form ${\rm B}_\g$ of $\g$ is negative definite on $\g_\R$, positive definite on $i \,\g_\R$, and $ {\rm B}_\g(\g_\R, i \g_\R) = 0$. By switching the sign on $\g_\R$ we thus obtain a positive definite inner product on $\g$.

The following result is known already, but we have decided to provide a proof in our context for convenience of the reader.

\begin{theorem}\cite{HS07,LauretDGA}\label{thm_mmsoliton}
A Lie group $G$ with Lie algebra $\g$ has, up to homotheties, at most one left-invariant metric $g$ satisfying the algebraic $\mm$-soliton equation
\[
 \mm(g)=\lambda\, g+g(D\cdot,\cdot)\,, \qquad \lambda\in\R, \quad D\in {\rm Der}(\g).
\] 
Moreover, if $G$ is not abelian, the Einstein-type equation 
		\[
		 	\mm(g)=\lambda\, g\,, \qquad \lambda\in\R
		\]
has a solution if and only if $G$ is semisimple, and in this case $\lambda < 0$ and a solution is given by the `canonical metric' induced by the Killing form of $\g$.
\end{theorem}

\begin{proof}
Fixing $g$ as background metric, the algebraic soliton equation is equivalent to
\[
	\mm_\mu = \lambda \, {\rm Id}_\g + D, \qquad D\in {\rm Der}(\g).
\]
From the proof of Theorem \ref{thm_mmflow} it follows that $\mu$ is an algebraic $\mm$-soliton if and only if it is a critical point of the functional $F(\mu) = \Vert\mm_\mu \Vert^2 / \Vert\mu\Vert^4$ (cf.~ also \cite[Proposition~3.2]{lauretJFA}).  Critical points for the norm of the moment map have been extensively studied in geometric invariant theory, and they enjoy a number of nice properties which are analogous to those satisfied by minimal vectors (i.e. the zeroes of the moment map). In particular, by the uniqueness result \cite[Corollary~9.4]{realGIT} two critical points in a fixed orbit ${\rm Gl}(\g) \cdot \mu$ must lie in fact in the same ${\rm O}(\g)$-orbit. Since brackets in the same ${\rm O}(\g)$-orbit correspond to isometric left-invariant metrics on $G$, this finishes the proof of the first claim.

Regarding the second claim, for the canonical metric $g$ on a semisimple Lie algebra with Cartan decomposition $\g = \kg \oplus \pg$ we have that $\ad_X$ is skew-symmetric for $X\in \kg$ and symmetric for $X\in \pg$. Thus if $\{ X_k\}$ is an orthonormal basis for $\g$ which is the union of basis for $\kg$ and $\pg$, then for $X\in \kg$ by \eqref{Mmu} we have
\begin{align*}
	\mm(g)(X,X) &= -\frac12 g(\mu(X,X_k), \mu(X,X_k)) + \frac14 g(\mu(X_j, X_k), X) g(\mu(X_j, X_k), X ) \\
		&= -\frac12 \tr \ad_X \ad_X^t + \frac14 g(X_k, \mu(X_j,X)) g(X_k, \mu(X_j,X)) \\
		& = -\frac14 \tr \ad_X \ad_X^t = \frac14 {\rm B}(X,X) = -\frac14 g(X,X),
\end{align*}
and analogously for $X\in \pg$. Thus, $\mm(g) = -\tfrac{1}{4} g$. 

Conversely, if $G$ non-abelian and admits a metric satisfying $\mm(g) = \lambda g$ then $\lambda < 0$, since 
\[
	\tr_g \mm(g) = -\frac14 \Vert \mu\Vert^2 \leq 0.
\]
In particular, $\mm(g) < 0$, and by Lemma \ref{lem_dotti} the group $G$ is semisimple.	
\end{proof}

\subsection{Applications to the HCF}
Theorem \ref{main_cxHCF}  follows by combining Corollary \ref{M-flow} with Theorem \ref{thm_mmflow}. We point out that the assumption on $G$ to be unimodular cannot in general dropped in order to have a long-time solution to the HCF (see an example in the next section). Regarding solitons and static metrics, Theorem \ref{main_cxsoliton} is an immediate consequence of  Corollary \ref{M-flow} and Theorem \ref{thm_mmsoliton}.

 Theorem \ref{main_cxHCF} yields the following result for complex parallelizable  manifolds.   

\begin{cor}
Let $(M,g_0)$ be a compact Hermitian manifold and let $g_t$ be the maximal solution to the HFC starting at $g_0$. Assume that the 
the holonomy group of the Chern connection of $g_0$ is trivial. Then the holonomy of the Chern connection of $g_t$ is trivial for any $t$, $g_t$ is immortal and satisfies 
$$
\partial_tg_t=-{\rm Ric}^{1,1}(g_t)\,. 
$$
\end{cor}
\begin{proof}
We recall that a compact complex manifold admits a Hermitian metric $g_0$ with trivial Chern holonomy  if and only if it is the compact quotient of a complex unimodular Lie group $G$ by a lattice $\Gamma$ and $g_0$ lifts to a left-invariant metric $\hat g_0$ on $G$ \cite{Boothby}. 

Let $\hat g_t$  be the left-invariant HCF solution on $G$ starting at $\hat g_0$ (see the discussion before Theorem \ref{main_cxHCF}). By Corollary \ref{M-flow} and Theorem \ref{thm_mmflow}, $\hat g_t$ is defined for all $t\in (-\varepsilon, \infty)$ for some $\varepsilon>0$. Moreover, $\hat g_t$ is still invariant under the action of $\Gamma \leq G$, and thus induces an HCF solution on $M = \Gamma\backslash G$ which by uniqueness coincides with $g_t$. The corollary now follows.
\end{proof}

Moreover, according to the analysis of solitons to the ${\rm M}$-flow in section 3.2, we say that a left-invariant Hermitian metric $g$ on a Lie group $G$ with a left-invariant complex structure is an {\em algebraic HCF soliton} if 
$$
K(g) = \lambda\, g + g( D \cdot, \cdot) + g(\cdot, D \cdot), \qquad \lambda \in \mathbb{R}, \quad D \in {\rm Der}(\g)\,.
$$
On a simply-connected Lie group $G$ an {\em algebraic HCF soliton} is a soliton in the usual way.  Indeed, if $\varphi_t \in {\rm Aut}(G)$ 
is the unique automorphism such that $d \varphi_t |_e = e^{-t D} \in {\rm Aut}(\g)$, then $g_t := (1- \lambda t ) \, \varphi_t^* g$ solves HCF. Proposition \ref{semiimplies} implies the following 
  
\begin{prop}
Every algebraic   {\rm HCF} soliton $g$ on a non-abelian complex unimodular Lie group is expanding and satisfies 
$$
K(g) = \lambda\, g + g( D \cdot, \cdot), \qquad \lambda \in \mathbb{R}, \quad D \in {\rm Der}(\g)\,.
$$
\end{prop}

In Proposition \ref{non-unimod} below we shall show that the unimodular assumption is necessary.

\section{HCF on $3$-dimensional complex Lie groups}

Recall that there exist exactly three non-abelian unimodular simply-connected complex Lie groups: ${\rm SL}(2,\C)$, 
the $3$-dimensional Heisenberg Lie group $H_3(\mathbb C)$ and a solvable Lie group $S_{3,-1}$ (see e.g. \cite{OV}). 

\medskip 

\subsection{ ${\rm SL}(2,\C)$. } This is a simple  Lie group and admits a left-invariant $(1,0)$-frame $\{Z_1,Z_2,Z_3\}$ such that 
\[
\mu(Z_1,Z_2)=Z_3\,,\quad \mu(Z_1,Z_3)=-Z_2\,,\quad \mu(Z_2,Z_3)=Z_1\,. 
\]
In matrix notation we can consider as frame $\{Z_1,Z_2,Z_3\}$  
$$
Z_1=\frac12\left(\begin{array}{cc} 
0 & i\\
i& 0
\end{array}
\right)\,,\quad 
Z_2=\frac12 \left(\begin{array}{cc} 
0 & 1\\
-1& 0
\end{array}
\right)\,,\quad 
Z_3=\frac12 \left(\begin{array}{cc} 
-i & 0\\
0 & i
\end{array}
\right)\,.
$$
A direct computation yields that the \lq\lq standard\rq\rq\,  metric
\[
g_{\rm std}= \zeta^1 \odot \bar \zeta^1 + \zeta^2 \odot \bar \zeta^2 + \zeta^3 \odot \bar \zeta^3\,,
\]
is static with $\lambda=-\frac32$, in accordance with  Theorem \ref{main_cxsoliton}. Here $\{\zeta^k\}$ is the dual frame to $\{Z_k\}$.\medskip

Let us call a left-invariant metric $g$ \emph{diagonal} if it can be written as  
\begin{equation}\label{eqn_gdiag}
	g = a \ \zeta^1 \odot \bar \zeta^1 + b \ \zeta^2\odot \bar \zeta^2 + c \ \zeta^3 \odot \bar \zeta^3, \qquad a,b,c > 0.
\end{equation}
Now, we study the behavior of  HCF on ${\rm SL}(2,\C)$ starting at a diagonal metric 
\[
	g_0 = a_0 \ \zeta^1 \odot \bar \zeta^1 + b_0 \ \zeta^2 \odot \bar \zeta^2 + c_0 \ \zeta^3 \odot \bar \zeta^3\,.
\]
For an arbitrary diagonal metric $g$ as in \eqref{eqn_gdiag} we have 
\[
		K(g) =  -\frac{{-a^2+b^2+c^2}}{2bc} \, \zeta^1 \odot \bar \zeta^1 - \frac{a^2-b^2+c^2}{2ac} \, \zeta^2 \odot \bar \zeta^2 -\frac{a^2+b^2-c^2}{2ab} \, \zeta^3\odot \bar \zeta^3\,.
\]
Thus, the HCF flow is governed by the following ODE system 
\begin{align*}
& a^\prime =  \frac {-a^2+b^2+c^2}{2bc}\,,\quad 
 b^\prime =  \frac {a^2-b^2+c^2}{2ac}\,,\quad 
 c^\prime =  \frac {a^2+b^2-c^2}{2ab} \,,\\
&a(0)=a_0\,,\quad b(0)=b_0\,,\quad c(0)=c_0\,,
\end{align*}
which can be explicitly solved. Indeed, from the above equations we deduce 
$$
\begin{aligned}
\frac{ a^\prime}{a}+\frac{ b^\prime}{b} =  \frac{c}{ab}\,,\quad 
\frac{ a^\prime}{a}+\frac{ c^\prime}{c} =  \frac{b}{ac}\,,\quad 
\frac{ b^\prime}{b}+\frac{ c^\prime}{c} =  \frac{a}{bc} \,,
\end{aligned}
$$ 
which imply 
\begin{equation}\label{abc}
\begin{aligned}
(ab)^\prime=c\,,\quad 
{(ac)}^\prime=b\,,\quad 
{(bc)}^\prime=a \,.
\end{aligned}
\end{equation} 
By substituting the last equation in the first two we get  
$$
({(bc)^\prime} b)^\prime=c\,,\quad ({(bc)}^\prime c)=b\,,
$$
and 
$$
{(bc)}^\prime b=\gamma\,,\quad {(bc)}^\prime c=\beta\,,
$$
where $\beta$ and $\gamma$ are primitives of $b$ and $c$, respectively. This in turn implies
$$
\frac bc = \frac \gamma\beta\,,
$$
i.e.,  $\beta  \beta^\prime = \gamma  \gamma^\prime$. Therefore, it follows ${(\beta^2)^\prime}={(\gamma^2)^\prime}$. Moreover, by arguing in the same way, we have 
$$
{(\alpha^2)}^\prime={(\beta^2)}^\prime={(\gamma^2)}^\prime\,,
$$
where $\alpha$ is a primitive of  $a$, and hence we get 
$$
a\alpha=b\beta=c\gamma.
$$
On the other hand, from \eqref{abc}, it follows
$$
ab-a_0b_0=\gamma\,, \quad ac-a_0c_0=\beta \,,\quad bc-b_0c_0=\alpha\,,
$$
and 
$$
abc-a_0b_0c=\gamma c\,, \quad abc-a_0c_0b=\beta b\,,\quad abc-b_0c_0a=\alpha a\,.
$$
Finally, keeping in mind that $a\alpha=b\beta=c\gamma$, we have 
$$
\frac{a}{a_0}=\frac{b}{b_0}=\frac{c}{c_0}
$$
and the ODE system simplifies to
\[
	 a^\prime =-\frac {a_0^2}{2 b_0c_0}+ \frac {b_0} {2c_0} + \frac {c_0}{2b_0} = : A_0 ,\quad 
	 b^\prime =\frac 12 \frac {a_0}{c_0}-\frac 12 \frac {b_0^2} {a_0c_0} +\frac 12 \frac {c_0}{a_0} = : B_0\,,\quad 
	 c^\prime =\frac 12 \frac {a_0}{b_0}+\frac 12 \frac {b_0} {a_0} -\frac 12 \frac {c_0^2}{a_0b_0} =: C_0\,,
\]
with solution
\[
	a(t)= A_0\cdot t + a_0\,,\quad b(t)= B_0\cdot t +b_0 \,,\quad c(t)=C_0\cdot t +c_0\,.
\]

\medskip 

\subsection{ $H_3(\mathbb C)$} This is the $2$-step nilpotent Lie group defined by
\[
	H_3(\mathbb C)=\left\{\left[\begin{smallmatrix}1&z_1&z_3\\0&1&z_2\\0&0&1\end{smallmatrix}\right]\mid z_1,z_2,z_3 \in \mathbb C\right\}\,.
\]
The group admits a left-invariant $(1,0)$-frame $\{Z_1,Z_2,Z_3\}$ such that 
\begin{equation}\label{muH3}
\mu=\zeta^1\wedge \zeta^2\otimes Z_3+\bar \zeta^1\wedge \bar \zeta^2\otimes \bar Z_3\,,
\end{equation}
where $\{\zeta^1,\zeta^2,\zeta^3\}$ is the dual frame of $\{Z_1,Z_2,Z_3\}$ and $\mu$ is the Lie bracket on $\mathfrak{h}_3(\C)$. 
\begin{prop}
Any left-invariant Hermitian metric on $H_3(\mathbb C)$ is a soliton to the {\rm HCF}. 
\end{prop}
\begin{proof} 
Let $g$ be a left-invariant  Hermitian metric on $H_3(\mathbb C)$. 
We can find a unitary frame
$\{W_1,W_2,W_3\}$ of $g$ such that 
$$
W_1\in \langle Z_1,Z_2,Z_3\rangle\,,\quad  W_2\in \langle Z_2,Z_3\rangle\,,\quad W_3\in \langle Z_3\rangle \,,
$$
where $\{Z_1,Z_2,Z_3\}$ is the left-invariant $(1,0)$-frame of \eqref{muH3}. With respect to this new frame $\mu$ can be written as 
$$
\mu=a \alpha^1\wedge \alpha ^2\otimes W_3+\bar a\bar \alpha^1\wedge \bar \alpha^2\otimes \bar W_3\,,
$$
for some constant $a\in \mathbb C$, with $\{\alpha^1,\alpha^2,\alpha^3\}$ dual frame to $\{W_k\}$. Then, from \eqref{CX}, we get 
$$
K=-\frac12 |a|^2\alpha^1\otimes \bar\alpha^1-\frac12 |a|^2\alpha^2\otimes \bar\alpha^2+\frac12 |a|^2\alpha^3\otimes \bar\alpha^3\,.
$$
Moreover, let $D={\rm diag}(\lambda_1,\lambda_2,\lambda_3)$ be a diagonal automorphism of $\mathfrak h_{3}(\C)$. Then, for any $X=x_iZ_i$ and $Y=y_kZ_k$ in $\mathfrak h_{3}(\C)$, we have 
$$
D\mu(X,Y)-\mu(DX,Y)-\mu(X,DY)=(\lambda_3-\lambda_1-\lambda_2)(x_1y_2-x_2y_1)Z_3
$$
and $D$ is a derivation if and only if 
$$
\lambda_3=\lambda_1+\lambda_2\,.
$$
Therefore, if we take 
$$
\lambda=- \frac{3}{2}|a|^2\,,
$$
then $K-\lambda I$ is a derivation of $\mathfrak h_{3}(\C)$ and $g$ is a soliton to the HCF. 
\end{proof}
%
%
%
%
%
\medskip

\subsection{$S_{3,-1}$} This is a $2$-step solvable Lie group whose Lie bracket can be written in terms of a suitable $(1,0)$-frame $\{Z_1,Z_2,Z_3\}$ as  
\begin{equation}\label{S3}
\mu=\zeta^1\wedge\zeta^2\otimes Z_2-\, \zeta^1\wedge\zeta^3\otimes Z_3 +\bar \zeta^1\wedge\bar \zeta^2\otimes\bar Z_2-\bar \zeta^1\wedge\bar \zeta^3\otimes\bar Z_3\,.
\end{equation}
The group belongs to a family $S_{3,\lambda}$ having structure equations
\[
\mu=\zeta^1\wedge\zeta^2\otimes Z_2+\lambda\, \zeta^1\wedge\zeta^3\otimes Z_3 +\bar \zeta^1\wedge\bar \zeta^2\otimes\bar Z_2+\bar \lambda\,\bar \zeta^1\wedge\bar \zeta^3\otimes\bar Z_3\,,
\]
with $\lambda \in \C$ and $0<|\lambda|\leq 1$, and it is the only unimodular group in that family. 

\begin{prop}\label{prop_S_31}
A left-invariant Hermitian metric $g$ on $S_{3,-1}$ is an algebraic HCF soliton if and only if $g(Z_2, \bar Z_3)=0$. 
\end{prop}

\begin{proof}
Given $g$ we can find a unitary frame $\{W_1,W_2,W_3\}$ such that 
$$
W_1 \in  \left< Z_1,Z_2,Z_3\right>\,,\quad W_2 \in \left<Z_2, Z_3 \right>\,,\quad W_3 \in \left<Z_3 \right>\,.
$$
With respect to this new frame, we have
$$\begin{aligned}
\mu(W_1,W_2)=sW_2+aW_3\,,\quad \mu(W_1,W_3)=-sW_3\,,\quad \mu(W_2,W_3)=0\,,
\end{aligned}$$
for some $s,a\in \C$ with $s\neq 0$. Note that,
$$
g(Z_2,Z_{\bar 3})=0 \iff a=0\,.  
$$
Now, with respect to the frame $\{W_1,W_2,W_3\}$, the matrix of $K_g$ is given by 
\begin{equation}
\begin{aligned}\label{Kexample}
K_g=\frac 12\left( \begin{matrix}-|a|^2-2|s|^2 & 2a\bar s & 0 \\
2\bar as& -|a|^2 & 0\\
0 & 0& |a|^2
\end{matrix}\right)\,,
\end{aligned}
\end{equation}
and we are looking for a derivation $D$ such that 
\begin{equation}\label{soliton_condition}
K_g = \lambda I +D\,,\quad D^t=D\,. 
\end{equation}
Setting
$$
DW_k= D_{ik} W_i\,,
$$
from the structure equations, we have 
$$
D\mu(W_2,W_3)-\mu(DW_2,W_3)-\mu(W_2,DW_3)=sD_{12}W_3+D_{13}(sW_2+aW_3)=0\,,
$$
and hence
$$
D_{12}=D_{13}=0\,,
$$ 
since $s\neq 0$. Similarly,
$$
D\mu(W_1,W_3)-\mu(DW_1,W_3)-\mu(W_1,DW_3)=-sD_{13}W_1-2sD_{23}W_2+(sD_{11}-aD_{23})W_3=0\,,
$$
which implies 
$$
D_{11}=D_{23}=0\,,
$$
and hence $D$ is diagonal. Finally, from \eqref{Kexample} we have  $a=0$ and $g(Z_2,\bar Z_3)=0$.\smallskip
 
Let us now assume $g(Z_2,\bar Z_3)=0$. Then, $a=0$ and \eqref{Kexample} reduces to 
$$
\begin{aligned}
K_g=\left( \begin{matrix}-|s|^2 & 0 & 0 \\
0& 0 & 0\\
0 & 0& 0
\end{matrix}\right)=-|s|^2\,I+D\,,
\end{aligned}
$$  
where  $D={\rm diag}(0,|s|^2,|s|^2)$. Finally, since 
$$
D\mu(W_1,W_2)-\mu(DW_1,W_2)-\mu(W_1,DW_2)=0\,,
$$
$D$ is a derivation and $g$ is an algebraic soliton.  
\end{proof}

\subsection{$S_{3,\lambda}$ with $\lambda>0$} In this section we study the HCF on a family of non-unimodular complex solvable Lie groups of dimension $3$.

\begin{prop}\label{non-unimod} Let $\lambda$ be a positive real number. Then, any diagonal left-invariant metric $g$ on $S_{3,\lambda}$ is a shrinking algebraic soliton.
\end{prop}

\begin{proof}
Let us consider a left-invariant diagonal metric on $S_{3,\lambda}$ given by
$$ 
g:= a_0\,\zeta^1\odot\bar\zeta^1+b_0\,\zeta^2\odot\bar\zeta^2+ c_0\,\zeta^3\odot\bar\zeta^3\,,\quad a_0,b_0, c_0>0\,.
$$
By means of \eqref{CX2}, a direct computation yields that
$$
K_g=\left(\begin{matrix}
{\lambda} & 0 & 0\\
0 & -\frac{b_0}{ a_0}(1+\lambda) & 0\\
0 & 0 & -\frac{c_0}{a_0}(\lambda+\lambda^2)
\end{matrix}\right)\,.
$$
On the other hand, arguing in the same way as Proposition \ref{prop_S_31}, one can prove that
$$
D:=K_g-c\,I
$$
is a derivation of the Lie algebra $\mathfrak s_{3,\lambda}$ if and only if $D_{11}=0$. Thus, setting
$$
c:={\lambda}
$$
the claim follows.
\end{proof} 
This in particular shows that HCF may develop finite-time singularities on complex Lie groups which are not unimodular. 
 
\section{Static left-invariant metrics on nilpotent Lie groups}
In this section we focus on nilpotent Lie groups $N$ equipped with a left-invariant Hermitian structure.  In \cite{EFV} it is proved that left-invariant pluriclosed metrics on non-abelian $2$-step nilpotent Lie groups  are never static with respect to the PCF. In the following, we generalize this result to a class of flows including the HCF.  

\medskip 
Given $x=(x_1,x_2,x_3,x_4)\in \R^4$ and a Hermitian metric $g$, let us set 
$$
K^{x}(g)=S-x_1Q^1-x_2Q^2-x_3Q^3-x_4 Q^4\,.
$$
Then, we can consider the geometric flow 
$$
\partial_tg_t=-K^{x}(g_t)\,,\qquad g_{t|t=0}=g_0\,.
$$
In this way,
\begin{itemize}
\item  for $x =( 1/2,- 1/4,- 1/2 , 1)$ the flow corresponds to the HCF;

\vspace{0.2cm}
\item for $x =( 1,0,0 ,0)$ the flow corresponds to the PCF \eqref{PCF};

\vspace{0.2cm}
\item for $x =( 0,-1/2,0 ,0)$ the flow corresponds to the modified HCF \eqref{yuriflow}. 
\end{itemize}
\begin{lemma}\label{theo_static}
Fix $x\in \R^4$ such that 
\[
	x_1\leq 1\,,\quad x_2\,,x_3\leq 0\,,\quad x_1+x_2 > 0 \,,\quad x_3+x_4\geq 0\,.
\] 
Let $G$ be a Lie group with a left-invariant complex structure $J$. Assume $\mathfrak{z}\cap J(\mathfrak{z})\neq 0$, where $\mathfrak{z}$ is the center of the Lie algebra of $G$. Then every left-invariant Hermitian metric $g$ on $G$ such that 
\begin{equation}\label{Kx}
K^{x}(g)=\lambda\,g\,,\qquad s\leq 0,
\end{equation}
is K\"ahler Ricci-flat.      
\end{lemma}
\begin{proof}
Let $g$ be a Hermitian metric and set $q^i={\rm tr}_{g}\,Q^i$. With respect to a unitary frame, we have
\begin{align*}
	&Q^1_{i\bar i}=T_{ik\bar m}T_{\bar i\bar k m}\,,\quad \quad 
	Q^2_{i\bar i}=T_{\bar k\bar m i}T_{km \bar i}\,,\\
	&Q^3_{i\bar i}=T_{ik\bar k}T_{\bar i\bar m m}\,,\quad \quad 
	Q^4_{i\bar j}=\frac12(T_{mk\bar k}T_{\bar m\bar ii} +T_{\bar m\bar kk}T_{mi\bar i})\,, \\
	&q^1=q^2=\|T\|^2\,,\qquad q^3=q^4=\|w\|^2\,,
\end{align*}
where $w_i = g^{j\bar k} T_{ij\bar k}$. Let now $g$ be left-invariant. By the assumption $\mathfrak{z}\cap J(\mathfrak{z})\neq 0$, there exists a left-invariant unitary frame $\{Z_i\}$ on $(G,g)$  such that $Z_1\in \mathfrak{z} \otimes \mathbb C$. Then, by the formulas in Section \ref{firstsection}, we have  
\[
	S_{1\bar 1}=\mu_{k\bar r}^{\bar 1} \mu_{\bar k r}^1,
\]
and 
\begin{align*}
	&Q^1_{1\bar 1}=\mu_{k\bar r}^{\bar 1}\mu_{\bar kr}^{1}\,,\quad \quad Q^2_{1\bar 1}=\mu_{\bar k\bar r}^{\bar 1}\mu_{kr}^1\,,\\
	&Q^3_{1\bar 1}=\mu_{k\bar k}^{\bar 1}\mu_{\bar r r}^1\,,\quad \quad Q^4_{1\bar 1}=0\,.
\end{align*}
Therefore, 
$$
	K^x_{1\bar 1}=\mu_{k\bar r}^{\bar 1} \mu_{\bar k r}^1-x_1\mu_{k\bar r}^{\bar 1}\mu_{\bar kr}^{1}-x_2\mu_{\bar k\bar r}^{\bar 1}\mu_{kr}^1-x_3\mu_{k\bar k}^{\bar 1}\mu_{\bar r r}^1
$$
and by the assumption on $x$ it follows $K^x(Z_1,\bar Z_1)\geq 0$. Assume now that $g$ further satisfies \eqref{Kx}.  
Since $n\lambda={\rm tr}_gK^x=nK^{x}(Z_1,Z_{\bar 1})$, we have that ${\rm tr}_gK^x\geq 0$. Thus,  
\[
	0 \leq {\rm tr}_gK^x=s-x_iq^i=s-(x_1+x_2)q^1-(x_3+x_4)q^3 \leq 0,
\]
since $s\leq 0$. Hence, we must have an equality and $q^i=0$. This implies $T=0$ and $Q^i=0$ for all $i$, from which $g$ is K\"ahler. Also, $\lambda = 0$, thus $S(g) = K^x(g) + x_i Q^i(g) = 0$ and $g$ is Ricci flat. 
\end{proof}

\begin{rem}
{\em The assumptions in Theorem \ref{theo_static} imply in particular that $\mathfrak{z} \neq 0$. Notice that condition $\mathfrak{z}\cap J	\mathfrak{z}\neq 0$ cannot be dropped in general dropped, as the examples of HCF static left-invariant metrics on ${\rm SL}(2,\C)$ show.
} 
\end{rem}

Theorem \ref{theo_static} implies the non-existence of a static Hermitian metric on non-abelian nilpotent Lie groups satisfying $\mathfrak{z} \cap J(\mathfrak{z})\neq 0$. 

\begin{prop}\label{lastcor}
Fix $x\in \R^4$ such that 
$$
x_1\leq 1\,,\quad x_2\,,x_3\leq 0\,,\quad x_1+x_2 > 0\,,\quad x_3+x_4\geq 0\,.
$$ 
Let $N$ be a non-abelian nilpotent Lie group with a left-invariant Hermitian structure $(g,J)$.  Assume that $\mathfrak{z}\cap J\mathfrak{z}\neq 0$, 
then there aren't left-invariant Hermitian metrics on $(G,J)$ satisfying the static equation 
$$
K^{x}(g)=\lambda\,g\,.
$$

\end{prop}

\begin{proof}
Since $N$ is nilpotent, the Chern scalar curvature of every left-invariant Hermitian metric on $N$  vanishes (see e.g. \cite[Proposition 2.1]{lauretval}). Hence Lemma \ref{theo_static} and the theorem of Benson and Gordon \cite{BG} imply the statement. 
\end{proof}

Note that Proposition \ref{lastcor} implies the  already known result about the non-existence of left-invariant pluriclosed metrics on nilpotent Lie groups static with respect to the PCF \cite{EFV}, since the pluriclosed condition forces 
$J\mathfrak{z}=\mathfrak{z}$ (see \cite[Proposition 3.1]{EFV}).   Moreover, the proposition can be applied to the HCF and we have 

\begin{cor}
Let $N$ be a non-abelian nilpotent Lie group with a left-invariant Hermitian structure $(g,J)$, and assume that 
$\mathfrak{z}\cap J\mathfrak{z}\neq 0$.
Then, there are no left-invariant Hermitian metrics on $(N,J)$ which are static with respect to HCF.
\end{cor}

\begin{rem}\rm The assumption $\mathfrak{z}\cap J\mathfrak{z}\neq 0$ is not always satisfied on a nilpotent Lie group. Indeed, the nilpotent Lie algebras with structure equations given by
$$
de^1=de^2=0\,,\quad de^3=e^{12}\,,\quad de^4=e^{13}\,,\quad de^5=e^{23}\,,\quad de^6=e^{14}+e^{25}
$$ 
or 
$$
de^1=de^2=de^3=0\,,\quad de^4=e^{13}\,,\quad de^5=e^{23}\,,\quad de^6=e^{14}+e^{25}
$$
and complex structures $Je^1=e^2$, $Je^3=e^6$, $Je^4=e^5$ do not satisfy the above hypothesis (see e.g. \cite{salamon}).
\end{rem}

\section{Evolution of Hermitian metrics compatible with abelian complex structures}
A left-invariant complex structure $J$ on a Lie group $G$ with Lie algebra $\g$ is said to be {\em abelian} if $\g^{1,0}$ is an abelian Lie algebra. In particular, for any left-invariant abelian complex structure $J$ on $G$, the center of $\g$ is $J$-invariant (\cite[Lemma 2.1]{ABD}).\smallskip

%
%
%
%
%

The following proposition is about the existence of static metrics compatible with abelian complex structures: 
\begin{prop}
Let $G$ be a unimodular Lie group with an abelian complex structure. Assume that the center of $G$ is not trivial.  Then $(G,J)$ does not admit any static metric unless it is abelian.
\end{prop}
\begin{proof}
We recall that the Chern scalar curvature of a left-invariant abelian balanced Hermitian structure is always vanishing (see \cite{{Vezzoni}}). Since the center of $\g$ is $J$-invariant and  non-trivial, the assumptions of Lemma \ref{theo_static} are satisfied and every static metric on $(G,J)$ is K\"ahler.  Finally we recall that a unimodular non-abelian Lie group with an abelian complex structure  does not admit any left-invariant K\"ahler metric \cite[Theorem 4.1]{ABD} and the claim follows. 
\end{proof}

In the following theorem we consider left-invariant balanced metrics compatible with abelian complex structures. We recall that in general a Hermitian metric  is {\em balanced} if its fundamental form is coclosed.  Balanced Hermitian metrics on Lie algebras with abelian complex structures are studied in \cite{av}. 

\begin{theorem}\label{abelian}
Let $G$ be a unimodular Lie group with an abelian complex structure $J$. A  left-invariant Hermitian metric $g$ on $(G,J)$ is balanced if and only if $k$ and the Riemannian scalar curvature coincide and  in the balanced case  we have $K={\rm Ric}^{1,1}$. 
Furthermore, the parabolic flow $\tfrac{d}{dt}\,g_t=-{\rm Ric}^{1,1}(g_t)$ specified by the $(1,1)$-component of the Ricci tensor has always a long-time solution for every left-invariant initial Hermitian metric.   
\end{theorem}

\begin{proof}
We compute $K$ and ${\rm Ric}^{1,1}$ of a Hermitian left-invariant metric $g$ on a unimodular  Lie group $G$ with an abelian complex structure. 
Fix as usual a unitary left-invariant frame $\{Z_1\,,\dots,Z_n\}$ and let $\mu$ the Lie bracket on the Lie algebra of $G$. Since $\mu_{ij}^r=0$, formulae in section \ref{firstsection}  directly imply 
$$
\begin{aligned}
Q^1_{i\bar j}=&\,\mu_{i\bar r}^{\bar k}\mu_{\bar j r}^{k}-\mu_{i\bar r}^{\bar k}\mu_{\bar kr}^{j}-\mu_{k\bar r}^{\bar i}\mu_{\bar j r}^{k}+\mu_{k\bar r}^{\bar i}\mu_{\bar kr}^{j} \, ,\\
Q^2_{i\bar j}=&\, \mu_{\bar ki}^{r}\mu_{k\bar j}^{\bar r}-\mu_{\bar ki}^{r}\mu_{r\bar j}^{\bar k}-\mu_{\bar ri}^{k}\mu_{k\bar j}^{\bar r}+\mu_{\bar ri}^{k}\mu_{r\bar j}^{\bar k}=2\mu_{\bar ki}^{r}\mu_{k\bar j}^{\bar r}-\mu_{\bar ki}^{r}\mu_{r\bar j}^{\bar k}-\mu_{\bar ri}^{k}\mu_{k\bar j}^{\bar r} \, ,\\
Q^3_{i\bar j}=&\,\mu_{k\bar k}^{\bar i}\mu_{\bar r r}^j \, ,\\
2Q^4_{i\bar j}=&\, -\mu_{k \bar k}^{\bar r}\mu_{\bar ri}^{j} + \mu_{k \bar k}^{\bar r} \mu_{\bar ji}^{r} -\mu_{\bar kk}^{r} \mu_{r \bar j}^{\bar i} + \mu_{\bar kk}^{r} \mu_{i \bar j}^{\bar r} \,.
\end{aligned}
$$
Furthermore the unimodular assumption together with the abelian condition and the Jacobi identity imply  
\begin{equation}\label{stramba}
\mu_{i\bar k}^{\bar r}\mu_{\bar rj}^l=\mu_{j\bar k}^{\bar r}\mu_{\bar ri}^l\
\end{equation}
and the formulae above simplify to 
$$
\begin{aligned}
Q^1_{i\bar j}=&\,\mu_{i\bar r}^{\bar k}\mu_{\bar j r}^{k}-\mu_{i\bar r}^{\bar k}\mu_{\bar kr}^{j}-\mu_{k\bar r}^{\bar i}\mu_{\bar j r}^{k}+\mu_{k\bar r}^{\bar i}\mu_{\bar kr}^{j} \, ,\\
Q^2_{i\bar j}=&\, 2\mu_{\bar ki}^{r}\mu_{k\bar j}^{\bar r} \, ,\\
Q^3_{i\bar j}=&\,\mu_{k\bar k}^{\bar i}\mu_{\bar r r}^j \, ,\\
2Q^4_{i\bar j}=&\, -\mu_{k \bar k}^{\bar r}\mu_{\bar ri}^{j} + \mu_{k \bar k}^{\bar r} \mu_{\bar ji}^{r} -\mu_{\bar kk}^{r} \mu_{r \bar j}^{\bar i} + \mu_{\bar kk}^{r} \mu_{i \bar j}^{\bar r} \,.
\end{aligned}
$$

Hence we have 
$$
\begin{aligned}
K_{i\bar j}=&-\mu_{\bar ki}^r\mu_{k\bar j}^{\bar r}+\mu_{k\bar r}^{\bar i} \mu_{\bar k r}^j+\mu_{k\bar k}^r\mu_{r\bar j}^{\bar i}-\mu_{k\bar k}^{\bar r}\mu_{\bar ri}^{j}-\frac12 (\mu_{i\bar r}^{\bar k}\mu_{\bar j r}^{k}-\mu_{i\bar r}^{\bar k}\mu_{\bar kr}^{j}-\mu_{k\bar r}^{\bar i}\mu_{\bar j r}^{k}+\mu_{k\bar r}^{\bar i}\mu_{\bar kr}^{j})\\
&+\frac12 \mu_{\bar ki}^{r}\mu_{k\bar j}^{\bar r}+\frac12 \mu_{k\bar k}^{\bar i}\mu_{\bar r r}^j-\frac12 (-\mu_{k \bar k}^{\bar r}\mu_{\bar ri}^{j} + \mu_{k \bar k}^{\bar r} \mu_{\bar ji}^{r} -\mu_{\bar kk}^{r} \mu_{r \bar j}^{\bar i} + \mu_{\bar kk}^{r} \mu_{i \bar j}^{\bar r})\\
\end{aligned}
$$
and a direct computation yields 
\begin{equation}\label{UnimAbel}
K_{i\bar j} = \frac 12 \left( -\mu_{\bar ki}^r\mu_{k\bar j}^{\bar r}+\mu_{k\bar r}^{\bar i} \mu_{\bar k r}^j -\mu_{i\bar r}^{\bar k}\mu_{\bar j r}^{k}+\mu_{k\bar k}^{\bar i}\mu_{\bar r r}^j-\mu_{k \bar k}^{\bar r} \mu_{\bar ji}^{r} - \mu_{\bar kk}^{r} \mu_{i \bar j}^{\bar r}\right) \,.
\end{equation}
Let us denote now by $D$ the Levi-Civita connection of $g$, $\Gamma_{ij}^k$ the Christoffel symbols of $D$ and 
$$
\begin{aligned}
{\rm Ric}_{i\bar j} =& \Gamma_{r \bar r}^{l}\Gamma_{il}^{j}+\Gamma_{r \bar r}^{\bar l}\Gamma_{i\bar l}^{j} -\Gamma_{i \bar r}^{l}\Gamma_{rl}^{j}-\Gamma_{i \bar r}^{\bar l}\Gamma_{r\bar l}^{j}+\Gamma_{\bar r r}^{l}\Gamma_{i l}^{j}+\Gamma_{\bar r r}^{\bar l}\Gamma_{i \bar l}^{j}-\Gamma_{ir}^{l}\Gamma_{\bar r l}^{j}-\mu_{i\bar r}^k \Gamma_{k r}^j - \mu_{i\bar r}^{\bar k}\Gamma_{\bar k r}^j\\
=&(\Gamma_{r \bar r}^{l}+\Gamma_{\bar r r}^{l} )\Gamma_{il}^{j}+(\Gamma_{r \bar r}^{\bar l}+\Gamma_{\bar r r}^{\bar l})\Gamma_{i\bar l}^{j}-(\Gamma_{ir}^{l}+\mu_{i\bar l}^{\bar r})\Gamma_{\bar r l}^{j}-(\Gamma_{i \bar r}^{l}+\mu_{i\bar l}^r)\Gamma_{rl}^{j}-\Gamma_{i \bar r}^{\bar l}\Gamma_{r\bar l}^{j}
\end{aligned}
$$
the $(1,1)$-component of the Ricci tensor.  The abelian condition together with Kozoul's formula imply 
$$
\Gamma_{kr}^l= \frac 12(-\mu_{r \bar l}^{\bar k}+\mu_{\bar l k}^{\bar r})\,,\quad \Gamma_{\bar kr}^l=\frac 12(\mu_{\bar kr}^l-\mu_{r \bar l}^{k})\,,\quad \Gamma_{\bar k r}^{\bar l}=\frac 12 (\mu_{\bar kr}^{\bar l}+\mu_{l\bar k}^{\bar r})\,,
$$
$$
\Gamma_{k\bar r}^l=\frac 12 (\mu_{k\bar r}^l+\mu_{\bar lk}^r)\,,\quad \Gamma_{k\bar r}^{\bar l}=\frac 12 (\mu_{k\bar r}^{\bar l}-\mu_{\bar rl}^{\bar k})\,.
$$
In paritucular since $G$ is unimodular we have 
$$
\Gamma_{r \bar r}^{l}=\Gamma_{r \bar r}^{\bar l}=0\,,\quad \Gamma_{ir}^{l}+\mu_{i\bar l}^{\bar r} = -\frac 12(\mu_{r \bar l}^{\bar i}+\mu_{\bar l i}^{\bar r})\,,\quad \Gamma_{i \bar r}^{l}+\mu_{i\bar l}^r=\frac 12 (\mu_{i\bar r}^l-\mu_{\bar li}^r)\,.
$$
and the formula of the Ricci tensor simplify to 
$$
\begin{aligned}
{\rm Ric}_{i\bar j}= \frac 14&\left(\mu_{r \bar l}^{\bar i}\mu_{\bar rl}^j+\mu_{\bar l i}^{\bar r}\mu_{\bar rl}^j-\mu_{r \bar l}^{\bar i}\mu_{l \bar j}^{r}-\mu_{\bar l i}^{\bar r}\mu_{l \bar j}^{r}\right.\\
&+\mu_{i\bar r}^l\mu_{l \bar j}^{\bar r}-\mu_{\bar li}^r\mu_{l \bar j}^{\bar r}-\mu_{i\bar r}^l\mu_{\bar j r}^{\bar l}+\mu_{\bar li}^r\mu_{\bar j r}^{\bar l}
-\left. \mu_{i\bar r}^{\bar l}\mu_{r\bar l}^j+\mu_{\bar rl}^{\bar i}\mu_{r\bar l}^j-\mu_{i\bar r}^{\bar l}\mu_{\bar jr}^l+\mu_{\bar rl}^{\bar i}\mu_{\bar jr}^l\right)\,.
\end{aligned}
$$
Finally, by using \eqref{stramba}, we obtain 
$$\begin{aligned}
{\rm Ric}_{i\bar j} = \frac 12\left(\mu_{r \bar l}^{\bar i}\mu_{\bar rl}^j-\mu_{\bar l i}^{\bar r}\mu_{l \bar j}^{r}-\mu_{\bar li}^r\mu_{l \bar j}^{\bar r}\right)\,.
\end{aligned}
$$
Therefore 
$$
K_{i\bar j}-{\rm Ric}_{i\bar j} = \frac 12 \left(\mu_{k\bar k}^{\bar i}\mu_{\bar r r}^j-\mu_{k \bar k}^{\bar r} \mu_{\bar ji}^{r} - \mu_{\bar kk}^{r} \mu_{i \bar j}^{\bar r}\right) 
$$
and  
$$
k-{\rm tr}_g{\rm Ric}= \frac 12 \left(\mu_{k\bar k}^{\bar i}\mu_{\bar r r}^i-\mu_{k \bar k}^{\bar r} \mu_{\bar ii}^{r} - \mu_{\bar kk}^{r} \mu_{i \bar i}^{\bar r}\right)=-\frac12 \mu_{\bar kk}^{r} \mu_{i \bar i}^{\bar r}\,.
$$
Since $G$ is unimodular, the metric $g$ is balanced if only if the sum $\sum_k\mu(Z_k,\bar Z_k)$ is vanishing and the first part of the claim follows. 

\medskip 
About the long-time existence of flow $\tfrac{d}{dt}\,g_t=-{\rm Ric}^{1,1}(g_t)$ we work as in the proof of Theorem \ref{main_cxHCF} by using the bracket flow argument. So we regard the flow as an evolution equation in the space $\mathcal C'$ consisting on brackets on the Lie algebra of $G$ satisfying the abelian assumption $\mu(J\cdot,J\cdot)=\mu(\cdot,\cdot)$. In this case the bracket flow equation writes as 
$$
\frac{d}{dt}	\mu_t=- \pi(K_{\mu_t})\mu_t\,,\quad \mu_{|t=0}=\mu_0\,,
$$
where for $\mu\in \mathcal C'$
$$
(K_\mu)_i^j= \frac 12\left(\mu_{r \bar l}^{\bar i}\mu_{\bar rl}^j-\mu_{\bar l i}^{\bar r}\mu_{l \bar j}^{r}-\mu_{\bar li}^r\mu_{l \bar j}^{\bar r}\right)
$$
Now if $\alpha$ is a real endomorphism of $\g$ that commutes with $J$, then

$$
-\pi( \alpha)\mu(Z_i,\bar Z_j) = (\alpha_i^k \mu_{k \bar j}^r + \alpha_{\bar j}^{\bar k} \mu_{i \bar k}^r - \mu_{i \bar j}^k \alpha_k^r) Z_r +(\alpha_i^k \mu_{k \bar j}^{\bar r} + \alpha_{\bar j}^{\bar k} \mu_{i \bar k}^{\bar r} - \mu_{i \bar j}^{\bar k} \alpha_{\bar k}^{\bar r}) \bar Z_r 
$$
and 

$$
\langle \alpha, K_\mu\rangle =2{\rm Re}(\alpha_{i}^{j}(K_{\mu})_{i}^j)=2 \sum_{i=1}^n Re \left\lbrace  \alpha_i^j(-\mu_{k \bar i}^{\bar r}\mu_{\bar kj}^{r}+\mu_{\bar kr}^{i} \mu_{k \bar r}^{\bar j} -\mu_{\bar i r}^{k}\mu_{j \bar r}^{\bar k})  \right\rbrace  \,.
$$
Moreover if $\theta\in \Lambda^2\g^*\otimes \g$ satisfies  $\theta(J\cdot,J\cdot)=\theta(\cdot,\cdot)$, then 
$$
\langle \mu,\theta\rangle  =2 {\rm Re} \left\lbrace \mu_{i \bar j}^k \theta_{\bar i j}^{\bar k} + \mu_{i \bar j}^{\bar k} \theta_{\bar i j}^k \right\rbrace \,,
$$
and so 
$$
\langle\pi(\alpha)\mu, \mu\rangle = 4 {\rm Re}  \left\lbrace \alpha_i^k \mu_{k \bar j}^r\mu_{\bar i j}^{\bar r}+ \alpha_i^k\mu_{k \bar j}^{\bar r}\mu_{\bar i j}^{r} - \alpha_k^r \mu_{i\bar j}^k\mu_{\bar ij}^{\bar r} \right\rbrace\,,
$$
which implies
$$
-\langle\pi(\alpha)\mu, \mu\rangle= - 2 \langle \alpha, P_\mu\rangle\,.
$$
In particular 
$$
\frac{d}{dt} \langle\mu, \mu\rangle  = -2\langle \pi(K_\mu)\mu, \mu\rangle = - 4\langle K_\mu,K_\mu\rangle
$$
and Theorem \ref{abelian} follows.
\end{proof}

\begin{rem}
{\em Notice that if a unimodular Lie group $G$ has a left-invariant $(1,0)$-frame $\{Z_1,\dots,Z_n\}$ such that $\mu(Z_i,\bar Z_i)=0$ for every fixed index $i$, then every diagonal left-invariant metric is balanced. }
\end{rem}

\begin{ex}\em
Let $\g$ be the 2-step nilpotent Lie algebra with structure equations 
$$ de^1=de^2=de^3=de^4=0\,,\quad de^5=e^{13}-e^{24}\,, \quad de^6=e^{14}+e^{23}$$ 
and $J$ the abelian complex structure given by $Je^1 = -e^2$, $Je^3=e^4$ and $Je^5=e^6$. If we set
$$Z_1 = \frac{1}{\sqrt2}(e_1 - i Je_1)\,, \quad Z_2 = \frac{1}{\sqrt2}(e_2 - i Je_2)\,, \quad Z_3 = \frac{1}{\sqrt2}(e_3 - i Je_3)\,,$$
then the bracket takes the expression
$$
\mu = -\sqrt 2 \, \zeta^1 \wedge \bar \zeta^2 \otimes \bar Z_3 -\sqrt 2 \,\bar  \zeta^1 \wedge \zeta^2 \otimes Z_3\,.
$$
In particular every diagonal metric $g = a\, \zeta^1 \odot \zeta^1 + b\, \zeta^2 \odot \zeta^2 + c\, \zeta^3 \odot \zeta^3$ is balanced and  \eqref{UnimAbel} implies 
$$
K(g)= - \frac{c}{b} \zeta^1 \odot \bar \zeta^1 - \frac{c}{a} \zeta^2 \odot \bar \zeta^2 +\frac{c^2}{ab} \zeta^3 \odot \bar \zeta ^3={\rm Ric}^{1,1}(g) \,.
$$
HCF starting from $g_0=\zeta^1 \odot \zeta^1 +  \zeta^2 \odot \zeta^2 +  \zeta^3 \odot \zeta^3$ is equivalent to the following system 
$$
 a^\prime = \frac{c}{b}\,, \quad  b^\prime = \frac{c}{b}\,, \quad  c^\prime = -\frac{c^2}{ab}\,, \quad a(0)=b(0)=c(0)=1\,.
$$
Hence
$$
g_t=\sqrt[3]{3t+1}\,\zeta^{1}\odot\zeta^{\bar 1}+\sqrt[3]{3t+1}\,\zeta^{2}\odot\zeta^{\bar 2}+\frac{1}{\sqrt[3]{3t+1}}\,\zeta^{3}\odot\zeta^{\bar 3}
$$
solves the problem. 
\end{ex}

\end{document}